\documentclass[11pt, a4paper]{amsart}

\usepackage[style=alphabetic]{biblatex}
\usepackage{amsmath, amsthm, amssymb,fullpage}

\usepackage{color}
\usepackage{xcolor}

\usepackage{enumitem}
\usepackage{hyperref}

\newcommand{\Fc}{\mathcal{F}}
\newcommand{\Gc}{\mathcal{G}}
\newcommand{\Pc}{\mathcal{P}}

\theoremstyle{plain}
\newtheorem{theorem}{Theorem}[section]
\newtheorem{corollary}[theorem]{Corollary}
\newtheorem{lemma}[theorem]{Lemma}
\newtheorem{proposition}[theorem]{Proposition}
\newtheorem*{claim}{Claim}

\theoremstyle{definition}
\newtheorem{definition}[theorem]{Definition}

\theoremstyle{remark}
\newtheorem{remark}[theorem]{Remark}

\numberwithin{equation}{section}

\renewcommand{\P}{\mathbb P}
\newcommand{\C}{\mathbb C}
\newcommand{\R}{\mathbb R}
\newcommand{\lam}{\lambda}
\newcommand{\eps}{\epsilon}

\newcommand{\abs}[1]{\left|#1\right|}
\newcommand{\pa}[1]{\left(#1\right)}
\newcommand{\norm}[1]{\left\|#1\right\|}

\newcommand{\Cc}{\mathcal C}

\newcommand{\Ll}{\mathcal L}

\usepackage{bbold}
\newcommand{\1}{\mathbb 1}
\newcommand{\sca}[1]{\left\langle#1\right\rangle}
\newcommand{\set}[1]{\left\{#1\right\}}

\newcommand{\D}{\mathbb D}
\newcommand{\N}{\mathbb N}
\newcommand{\B}{\mathbb B}

\usepackage{tikz-cd}

\DeclareMathOperator{\Leb}{Leb}
\DeclareMathOperator{\dist}{dist}
\DeclareMathOperator{\supp}{supp}

\DeclareMathOperator{\jac}{Jac}
\DeclareMathOperator{\diam}{diam}
\DeclareMathOperator{\lip}{Lip}

\newcommand{\id}{{\rm id}}

\newcommand{\FS}{{\text{\rm \tiny FS}}}

\newcommand{\ddc}{{dd^c}}

\newcommand{\dbar}{{\overline\partial}}
\newcommand{\ddbar}{{\partial\overline\partial}}

\newcommand{\h}{\mathbb{h}}
\renewcommand{\H}{\mathbb{H}}
\DeclareMathOperator{\ent}{Ent}
\newcommand{\Ent}[2]{\ent_{#1} ({#2})}

\usepackage{enumitem}
\usepackage{hyperref}

\subjclass[2010]{37F80, 37D35 (primary), 32U05, 32H50 (secondary)}

\keywords{Equilibrium states, Transfer operator,
Repelling periodic points, K-mixing}

\begin{document} 

\hyphenpenalty=10000

\title[Equilibrium states of  endomorphisms of  $\P^k$ I]{Equilibrium states of  endomorphisms of  $\P^k$ I:\\
Existence and properties}

\begin{author}[F.~Bianchi]{Fabrizio Bianchi}
\address{ 
CNRS, Univ. Lille, UMR 8524 - Laboratoire Paul Painlev\'e, F-59000 Lille, France}
  \email{fabrizio.bianchi$@$univ-lille.fr}
\end{author}

\begin{author}[T.C.~Dinh]{Tien-Cuong Dinh}
\address{National University of Singapore, Lower Kent Ridge Road 10,
Singapore 119076, Singapore}
\email{matdtc$@$nus.edu.sg }
\end{author}

\maketitle 

{\centering\small \emph{Dedicated to the memory of Professor Nessim Sibony}\par}

\begin{abstract}
We develop a new method, based on pluripotential theory, to
study the transfer (Perron-Frobenius) operator induced on $\P^k = \P^k (\C)$
 by a holomorphic endomorphism and a suitable continuous weight.
This method 
allows us to prove the existence and uniqueness of  
 the equilibrium state and conformal measure for very general weights
(due to Denker-Przytycki-Urba\'nski in dimension 1 and
Urba\'nski-Zdunik in higher dimensions, both in the case of H\"older continuous weights).
 We establish a number of
 properties of the equilibrium states, 
including 
 mixing, K-mixing, mixing of all orders, and
an equidistribution of repelling periodic
points.
Our analytic method replaces all distortion estimates on inverse branches with a unique, global,
estimate on dynamical currents, and allows us to reduce the 
dynamical questions to comparisons between currents and their potentials.
\end{abstract}

\setcounter{secnumdepth}{3}
\setcounter{tocdepth}{1}
\tableofcontents

\bigskip

\noindent
{\bf Notation.} 
Throughout the paper, $\P^k$ denotes the complex projective space of dimension $k$
endowed with the standard Fubini-Study form $\omega_\FS$. This is a K\"ahler $(1,1)$-form
normalized so that $\omega_\FS^k$ is a probability measure. We will use the metric and
distance $\dist(\cdot,\cdot)$ on $\P^k$ induced by $\omega_\FS$ and the standard ones
on $\C^k$ when we work on open subsets of $\C^k$. 
We denote by $\B_{\P^k}(a,r)$ (resp.\ $\B_r^k, \D(a,r), \D_r$) 
the ball of center $a$ and radius $r$ in $\P^k$  (resp.\ the ball of center 0 and radius $r$ in $\C^k$,
the disc of center $a$ and radius $r$ in $\C$, and the disc
of center $0$ and radius $r$ in $\C$). 
$\Leb$ denotes the standard Lebesgue measure on a
 Euclidean
 space or on a sphere. 
The oscillation $\Omega(\cdot)$,
the modulus of continuity $m(\cdot,\cdot)$,
 and the
semi-norms $\norm{\cdot}_{\log^p}$ 
of a function
are defined in Section
  \ref{ss:def}.
The currents $\omega_n$ and their dynamical potentials $u_n$ are introduced in Section
\ref{ss:bound-pot-push}.

The pairing $\langle \cdot,\cdot\rangle$ is used for the integral of a function with respect to a
measure or more generally the value of a current at a test form. If $S$ and
$R$ are two $(1,1)$-currents, we will write $|R|\leq S$ when $\Re (\xi R)\leq S$
for every function $\xi\colon\P^k\to\C$ with $|\xi|\leq 1$,
 i.e., all currents $S- \Re (\xi R)$ with $\xi$ as before are positive. Notice 
that this forces $S$ to be real and positive. We also write other inequalities such
as $|R|\leq |R_1|+|R_2|$ if $|R|\leq S_1+S_2$ whenever $|R_1|\leq S_1$ and $|R_2|\leq S_2$.
Recall that $d^c={i\over 2\pi}(\dbar -\partial)$ and $\ddc={i\over\pi}\ddbar$. 
The notations $\lesssim$ and $\gtrsim$ stand for inequalities up to a multiplicative constant.
The function identically equal to 1 is denoted by $\1$.
We also use the function $\log^\star (\cdot):=1+|\log (\cdot)|$. 

Consider a holomorphic endomorphism $f\colon\P^k\to\P^k$ of
 algebraic degree $d\geq 2$ satisfying the Assumption {\bf (A)} in the Introduction. 
Denote respectively by $T$, $\mu=T^k$, $\supp(\mu)$ the Green $(1,1)$-current, 
the measure of maximal entropy (also called
the Green measure or the
equilibrium measure), and the small Julia set of $f$. 
If $S$ is a positive closed $(1,1)$-current on $\P^k$, its dynamical potential is denoted by $u_S$ and is
defined in Section \ref{ss:dyn-pot}.
If $\nu$ is an invariant probability measure, we denote by $\Ent{f}{\nu}$ the metric
entropy of $\nu$ with respect to $f$.

We also consider a weight $\phi$ which is a
 real-valued
continuous function on $\P^k$. 
The
transfer operator (Perron-Frobenius operator) $\Ll=\Ll_\phi$ is introduced in the
Introduction together with the scaling ratio $\lambda= \lambda_\phi$, the conformal measure $m_\phi$,
the density function $\rho=\rho_\phi$, the equilibrium state $\mu_\phi=\rho m_\phi$, the pressure $P(\phi)$,
see also Section \ref{s:lambda-rho}. The measures $m_\phi$ and $\mu_\phi$ are probability measures.
The operator $L$ is a suitable modification of $\Ll$
and is introduced in Section \ref{s:equidistributions-L}.

\section{Introduction and results} \label{s:Intro}

Let $f\colon \P^k\to \P^k$ be a holomorphic endomorphism of 
the complex projective space $\P^k=\P^k (\C)$, with $k\geq 1$, of algebraic degree $d\geq 2$.
Denote by $\mu$ the unique measure of maximal entropy
for the dynamical system $(\P^k, f)$
\cite{lyubich1983entropy,briend2009deux,dinh2010dynamics,berteloot2001rudiments}. The support $\supp(\mu)$ of $\mu$ is called {\it the small Julia} set of $f$.
The measure $\mu$ corresponds to the equilibrium state of the system in the case without weight, i.e., when
the weight is zero. In this paper, we will consider 
the case where the weight, denoted by $\phi$, is not necessarily equal to zero. 
This problem has been studied 
for H\"older continuous
weights
using a geometric approach, 
in dimension 1, see, e.g., 
Denker-Przytycki-Urba\'nski 
\cite{przytycki1990perron,denker1991ergodic,denker1991existence,denker1996transfer}
and Haydn \cite{haydn1999convergence}
just to name a few,
and
in higher dimensions, see 
Szostakiewicz-Urba\'nski-Zdunik
 \cite{urbanski2013equilibrium,szostakiewicz2014stochastics}.
 We
will develop here an analytic method which will allow us to obtain more general and more quantitative results.
Many results are new even when
 for  $k=1$.

Throughout this paper, we make use of the 
following technical assumption for $f$:

\medskip\noindent
{\bf (A)} \hspace{1cm} the local degree of the iterate $f^n:=f\circ\cdots\circ f$ ($n$ times) satisfies
$$\lim_{n\to\infty} {1\over n} \log\max_{a\in\P^k}\deg(f^n,a) =0.$$

\medskip\noindent
Here, $\deg(f^n,a)$ is the multiplicity of $a$ as a solution of the
equation $f^n(z)=f^n(a)$. Note that generic endomorphisms of $\P^k$ satisfy
 this condition, 
 see \cite{dinh2010equidistribution}.
 Our study still holds under a weaker
 condition that the exceptional set of $f$ (i.e., the maximal proper analytic subset of $\P^k$ invariant
 by $f^{-1}$) is empty or more generally has no intersection with $\supp(\mu)$
 (in particular, this condition is superfluous in dimension 1). 
 However, this situation requires more technical conditions on the weight $\phi$.
 We choose not to present this case here in order to simplify the notation and focus on the main new ideas introduced in this topic.
Our 
 main
 goal in this paper is to prove the following theorem
(see 
 Theorem \ref{t:main-lam} and Section \ref{s:new-properties} for more precise statements).

\begin{theorem}\label{t:main}
Let $f$ be an endomorphism of $\P^k$ of algebraic degree $d\geq 2$ and
satisfying the Assumption {\bf (A)} above. Let $\phi$ be 
a real-valued $\log^q$-continuous function on $\P^k$, for some $q>2$, 
such that  $\Omega (\phi) :=\max \phi - \min \phi < \log d$. Then
$\phi$ admits a unique equilibrium state $\mu_\phi$, whose support is equal to  
the small Julia set of $f$. This measure $\mu_\phi$ is $K$-mixing
and mixing of all orders, and 
 repelling 
periodic 
 points of period $n$
 (suitably weighted) 
 are equidistributed with respect to 
 $\mu_\phi$ as $n$ goes to infinity.
Moreover, there is a unique conformal measure $m_\phi$
associated to $\phi$. We have $\mu_\phi=\rho m_\phi$ for some strictly positive continuous function $\rho$ on  $\P^k$
and the preimages of points by $f^n$ (suitably weighted)
 are equidistributed with respect to $m_\phi$ as $n$ goes to infinity.
\end{theorem}

We say that a function is $\log^q$-continuous if its oscillation
on a ball of radius $r$ is bounded by a constant times
$(\log^\star r)^{-q}$, 
see Section
\ref{ss:def}
 for details. See also Section 
\ref{s:equidistributions-L}
  for  the $K$-mixing and mixing of all orders.

An \emph{equilibrium state} as in the statement above is defined as follows, see for instance
 \cite{ruelle1972statistical,walters2000introduction,przytycki2010conformal}. 
Given a \emph{weight}, i.e., a real-valued continuous
function, $\phi$ as above, we define the \emph{pressure} of $\phi$
as
$$P(\phi) := \sup \big\{ \Ent{f}{\nu} + \langle\nu,\phi\rangle \big\},$$
where the supremum is taken over all Borel $f$-invariant probability measures $\nu$ and $\Ent{f}{\nu} $
denotes the metric entropy of $\nu$.
An equilibrium state for $\phi$ is then an invariant probability measure $\mu_\phi$ 
realizing a maximum in the above formula, that is, 
$$P(\phi) = \Ent{f}{\mu_\phi}+ \langle\mu_\phi,\phi\rangle.$$
On the other hand, a \emph{conformal measure} is defined as follows.
Define the \emph{Perron-Frobenius} (or \emph{transfer}) operator $\Ll$
with weight $\phi$ as (we often drop the index $\phi$ for simplicity)
\begin{equation} \label{e:L}
\Ll g(y):=\Ll_\phi g (y):= \sum_{x \in f^{-1}(y)} e^{\phi (x)} g(x),
\end{equation}
where $g\colon\P^k \to \R$ is a continuous test function and the
points $x$ in the sum are counted with multiplicity. 
A conformal measure is an eigenvector
for the dual operator $\Ll^*$ acting on positive measures.

Notice that, in the case where $\phi$ is H\"older continuous,
a part of
Theorem \ref{t:main} was established by
Urba\'{n}ski-Zdunik \cite{urbanski2013equilibrium}
(also under a genericity assumption for $f$),
 see also
\cite{przytycki1990perron,denker1991ergodic,denker1991existence,denker1996transfer} for
previous results in
dimension $k=1$. When $\phi$ is constant,
the operator $\Ll$ reduces to a constant times the push-forward
operator $f_*$ and we get $\mu_\phi=\mu$. For an account of the 
known
results in this case, see for instance \cite{dinh2010dynamics}.

\medskip

A reformulation of Theorem \ref{t:main} is the following:
given $\phi$ as in the statement, there exist a number $\lam >0$ and
a continuous function $\rho=\rho_\phi \colon \P^k \to \R$ 
such that, for every continuous function $g\colon \P^k\to \R$, 
the following uniform convergence holds:
\begin{equation}\label{e:intro-cv-rho}
\lam^{-n}\Ll^n  g (y) \to c_g  \rho
\end{equation}
for some constant $c_g$ depending on $g$. By duality,
this is equivalent to the convergence, uniform on probability measures $\nu$,
\begin{equation}\label{e:intro-cv-m}
\lam^{-n} (\Ll^*)^{n} \nu  \to m_\phi,
\end{equation}
where $m_\phi$ is a conformal measure associated to the weight
$\phi$.  The equilibrium state $\mu_\phi$ is then given by
$\mu_\phi = \rho m_\phi$, and we have $c_g = \langle m_\phi, g \rangle$. 

\medskip

To prove Theorem \ref{t:main}, in Section \ref{s:lambda-rho}
we develop a new and completely different approach with respect
to \cite{urbanski2013equilibrium} and to the previous studies in dimension 1.
As we will see
 in the second part of this work \cite{bd-eq-states-part2},
the flexibility of this method will allow
 for a more quantitative understanding
  of the convergences
  \eqref{e:intro-cv-rho} and \eqref{e:intro-cv-m}, 
  and for the direct establishment 
of several statistical properties of the equilibrium states.

The main idea of our method is the following. Let us just consider for 
now the case where both of the functions $g$ and $\phi$ are of class $\Cc^2$ (the
general case
 is technically quite involved and
 requires suitable approximations of $g$ and $\phi$ by $\Cc^2$ functions). 
Given such a function $g$, first we want to prove
that the ratio between the maximum and the minimum of 
$\Ll^n g$ stays bounded with $n$. This allows us 
to define the good scaling ratio $\lam$ and to get that the sequence
$\lambda^{-n}\Ll^n g$ is uniformly bounded. Next, we would like to
prove that this sequence is actually equicontinuous. This,  
together with other technical arguments, would imply the existence and uniqueness
of the limit function $\rho$.

In order to establish the above controls, we study the sequence
of $(1,1)$-currents given by $dd^c \Ll^n g $. First we prove that
suitably normalized versions of these currents are uniformly bounded
by a common positive closed $(1,1)$-current $R$.
This is the core of our method which replaces all controls on the distortion of
inverse branches of $f^n$ in the geometric method of \cite{urbanski2013equilibrium}
by a unique, global, and flexible estimate. Namely, for every $n\in \N$
we can get an estimate of the form
\begin{equation}\label{e:intro-ddc}
\Big| dd^c \frac{\Ll^n g }{c_n} \Big| \lesssim \sum_{j=0}^{\infty} \Big( \frac{e^{\Omega(\phi)}}{d} \Big)^j
\frac{ (f_*)^j \omega_\FS}{d^{(k-1)j}}
 \quad \text{with} \quad c_n:=\|g\|_{\Cc^2} \langle \omega_\FS^k,\Ll^n \1\rangle.
\end{equation}
Here, $\omega_\FS$
denotes the usual Fubini-Study form on $\P^k$ normalized so that $\omega_\FS^k$ is a probability measure.
Notice that the last infinite sum gives a
 key reason for the assumption $\Omega (\phi)< \log d$ made on the
 weight $\phi$ as the mass of the current $(f_*)^j \omega_\FS$ is equal to $d^{(k-1)j}$.

We will establish in Section \ref{s:preliminary} some general criteria, interesting in themselves,
which allow one to bound the oscillation of $c_n^{-1} \Ll^ng$ in terms of the oscillation
of the potentials of the current in the RHS of \eqref{e:intro-ddc}. 
This latter oscillation is actually controllable.
Assumption {\bf (A)} allows us to have a simple
control which makes the estimates less technical but such a control exists without Assumption {\bf (A)}.

Combining all these ingredients, 
the existence and uniqueness of the equilibrium state and conformal measure, as well as the equidistribution of preimages
and the equality $P(\phi)=\log \lam$,
 follow from standard arguments
that we recall in Sections
\ref{s:equidistributions-L} and
\ref{s:proof-t:main-further}
 for completeness. 
We also prove 
that the entropy of $\mu_\phi$ is larger than
$k \log d - \Omega (\phi) > (k-1)\log d$, and that all the Lyapunov 
exponents of $\mu_\phi$ are strictly positive,
 see Proposition \ref{p:large-entropy}.
This also leads to a lower bound for the Hausdorff dimension of
$\mu_\phi$.
In Section \ref{s:pern}
we establish
  the equidistribution of repelling periodic  points with respect to $\mu_\phi$,
   see Theorem
  \ref{t:equidistr-pern}, which 
  completes the proof
  of Theorem \ref{t:main}. This result is  
due to Lyubich \cite{lyubich1983entropy} (for $k=1$)
 and Briend-Duval
\cite{briend1999exposants} (for any $k\geq 1$)
 when $\phi=0$, and   
is  new even for $k=1$ otherwise.

\medskip

 In the second part of our study
\cite{bd-eq-states-part2},
we will prove that the Perron-Frobenius operator and its
complex
perturbations admit spectral gaps, and deduce several statistical properties of the equilibrium states
through a unified method.

\medskip\noindent
{\bf Outline of the organization of the paper.} 
In Section \ref{s:preliminary}, we introduce some useful notions
 and establish comparison principles for currents and potentials that 
 will be the technical key to prove Theorem
 \ref{t:main}.
 We also
 present the estimates on the sequence $f^n_*\omega_\FS$
 (and on their potentials)
  that we will need in the sequel.
 Section \ref{s:lambda-rho} is dedicated to the proof of Theorem \ref{t:main-lam}.
 For this purpose, we develop our method to get the uniform boundedness
 and equicontinuity for the sequence $\Ll^n g$, properly normalized, 
 that lead to the good definition 
 of the scaling ratio $\lambda$. 
Once this is done,
 we will complete the proof of
Theorem \ref{t:main} in Section \ref{s:new-properties}.

\medskip\noindent 
\textbf{Acknowledgements.}
The first author would like to thank the National University
of Singapore (NUS) for its support and hospitality during the visits where this project started and developed, and 
Imperial College London were he was based during the first part of this work.

This project has received funding from the European Union’s Horizon 2020
Research and Innovation Programme under the Marie Skłodowska-Curie grant agreement No
796004,
 the French government through the Programme Investissement d'Avenir (I-SITE ULNE / ANR-16-IDEX-0004 ULNE and 
LabEx CEMPI /ANR-11-LABX-0007-01) managed by the Agence Nationale de la Recherche,
the CNRS  through the program PEPS JCJC 2019,
and the NUS 
and MOE
through the grants C-146-000-047-001,
R-146-000-248-114, 
and
MOE-T2EP20120-0010.

\section{Dynamical potentials and some comparison principles}\label{s:preliminary}

\subsection{$\log^p$-continuous functions}\label{ss:def}

 We will use the following notations
throughout the paper.

\begin{definition}
Given a subset $U$ of $\P^k$ or $\C^k$ and a real-valued function $g\colon U \to \R$,
define the \emph{oscillation} $\Omega_U (g)$ of $g$ as
\[\Omega_U (g) := \sup g - \inf g\]
and  its continuity modulus $m_U(g,r)$ at distance $r$ as
\[m_U(g,r) := \sup_{x,y\in U\colon \dist(x,y)\leq r} |g(x)-g(y)|.\]
We may drop the index $U$ when there is no possible confusion. 
\end{definition}

\begin{definition}\label{defi_logp_cont}
The semi-norm  $\norm{\cdot}_{\log^p}$
is defined for every $p>0$ and $g\colon \P^k \to \R$ as
\[\norm{g}_{\log^p} := \sup_{a,b\in \P^k} |g(a)-g(b)| \cdot (\log^\star \dist(a,b))^p=
\sup_{r>0, a\in \P^k} \Omega_{\B_{\P^k}(a,r)}(g)\cdot (1+|\log r|)^p,\]
where $\B_{\P^k}(a,r)$ denotes the ball of center $a$ and radius $r$ in $\P^k$.
\end{definition}

 The following technical lemma will be used
in Section \ref{s:lambda-rho}.
 
\begin{lemma}\label{l:g-Cs-logp}
For every $\log^p$-continuous function $g\colon \P^k \to \R$, $p>0$, $s\geq 1$, 
and $0<\eps\leq 1$, there exist continuous functions $g_\eps^{(1)}$ and $g_\eps^{(2)}$ such that
$$g = g_\eps^{(1)} + g_\eps^{(2)},\qquad 
\|g_\eps^{(1)}\|_{\Cc^s} \leq 
c \norm{g}_{\infty} e^{ (1/\eps)^{1/p}}, \qquad \text{and} \qquad  
\|g_\eps^{(2)}\|_{\infty} \leq c \norm{g}_{\log^p} \eps,$$
where $c=c(p,s)$  is a positive constant independent of $g$ and $\epsilon$.
In particular, for every $n \geq 1$ there exist $g_n^{(1)}$ of class $\Cc^2$ and $g_n^{(2)}$ continuous 
such that
$$g = g_n^{(1)} + g_n^{(2)},\qquad \|g_n^{(1)}\|_{\Cc^2} \leq  c \norm{g}_{\infty}  e^{{1\over 2}n^{2/p}},
 \qquad \text{and} \qquad  \|g_n^{(2)}\|_{\infty} \leq c  \norm{g}_{\log^p} n^{-2}.$$
\end{lemma}

\begin{proof}
Clearly, the second assertion is a consequence of the first one by taking $\epsilon=2^pn^{-2}$
and replacing $c$ by $2^p c$.
 We prove now the first assertion.
Using a partition of unity, we can reduce the problem to the case where $g$ is
supported by the unit ball of an affine chart $\C^k\subset\P^k$.

Consider a smooth non-negative function $\chi$ with support in
 the unit ball of $\C^k$ whose integral with respect to the Lebesgue measure is 1. 
For $\nu>0$, consider the function $\chi_\nu(z):=\nu^{-2k}\chi(z/\nu)$
which has integral 1 and tends to the Dirac mass at 0 when $\nu$ tends to 0. 
Define an approximation of $g$ using the standard convolution operator $g_\nu:=g*\chi_\nu$,
and define $g_\eps^{(1)}:=g_\nu$ and $g_\eps^{(2)}:=g-g_\nu$.
We consider $\nu:=e^{-1/(M\eps)^{1/p}}$ for some constant $M>0$ large enough.
It remains to bound $\|g^{(1)}_\eps\|_{\Cc^s}$ and $\|g^{(2)}_\eps\|_\infty$. 

By standard properties of the convolution we have, for some constant $\kappa>0$,
$$\|g^{(2)}_\eps\|_\infty \lesssim m(g, \kappa\nu) 
\lesssim \norm{g}_{\log^p}(\log^\star \nu)^{-p} \lesssim \norm{g}_{\log^p}\epsilon$$
and, by definition of $g_\nu$,
$$\|g^{(1)}_\eps\|_{\Cc^s} \lesssim \norm{g}_\infty \norm{\chi_\nu}_{\Cc^s} \Leb (\B_\nu^k)\lesssim
\norm{g}_\infty \nu^{-s} \lesssim \norm{g}_\infty e^{(1/\epsilon)^{1/p}},$$ 
where we use the fact that $M$ is large enough. This ends the proof of the lemma.
\end{proof}

\subsection{Dynamical potentials}\label{ss:dyn-pot}
Let $T$ denote the Green $(1,1)$-current
of $f$. It is positive closed and of unit mass. Let $S$ be any positive closed
$(1,1)$-current of mass $m$ on $\P^k$. There is a unique function $u_S\colon \P^k \to \R \cup \{-\infty\}$ 
which is p.s.h.\ modulo $mT$
 and such that
\[S = mT + dd^c u_S \qquad \text{and} \qquad \langle \mu, u_S\rangle=0.\]
Locally, $u_S$ is the difference between 
a potential of $S$ and a potential of $mT$. We call it {\it the dynamical potential} of $S$.
Observe that the dynamical potential of $T$ is zero, i.e., $u_T=0$. 

Recall that 
$T$ has H\"older continuous potentials. 
So, $u_S$ is locally the difference
between a p.s.h.\ function and a H\"older continuous one.
  The dynamical potential of  $S$ behaves well
 under the
  push-forward and pull-back
operators associated to $f$. Indeed, because of the invariance properties of $T$, we have
\[f^* S = md\cdot T + dd^c  (u_S \circ f) \qquad \mbox{and} \qquad
f_* S = md^{k-1} \cdot T + dd^c (f_* u_S),\]
which, together with the invariance properties of $\mu$, imply
$$u_{f^*S}=u_S\circ f \qquad \text{and} \qquad  u_{f_*S}=f_*u_S.$$
We refer the reader to  \cite{dinh2010dynamics} for details. 
In this paper, we 
only need currents $S$ such that $u_S$ is continuous.

\subsection{Comparisons between currents and their potentials} 
A technical key point in the proof of our main theorem
 will be based on the following general 
idea:  if $u$ and $v$ are two functions on some domain in $\C^k$ 
such that $\abs{dd^c u}\leq dd^c v$,
 then $u$ inherits
some of the regularity properties of $v$.  This section
 is
 devoted to make this idea precise and 
quantitative for our purposes.
We start with the simplest occurrence of this fact
in the first case 
in terms of the sup-norm.

\begin{lemma}\label{l:compare-sup}
There exists a positive constant $A$ such that, for every 
 positive closed $(1,1)$-current $S_0$
 on $\P^k$ of mass $1$ and 
for every  positive closed $(1,1)$-current $S$ on $\P^k$ with $S\leq S_0$,
we have
$\Omega (u_S) \leq A + \Omega (u_{S_0})$, where $u_{S_0}$ and $u_S$
denote the dynamical potentials of $S_0$ and $S$, respectively.
\end{lemma}

\begin{proof}
We assume that $\Omega (u_{S_0})$ is finite, since otherwise the assertion trivially holds. 
Observe that the mass $m$ of $S$ is at most equal to 1 because $S\leq S_0$.
Recall that $u_{S}$ and $u_{S_0}$ satisfy
\[S =  mT + dd^c u_S,  \quad 
S_0 = T + dd^c u_{S_0}, \quad \langle \mu, u_S\rangle=0, \quad \mbox{ and } \quad  \langle \mu, u_{S_0}\rangle=0.\]
The last identity implies that $\sup u_{S_0}$ is non-negative.

We first prove that $u_S$ is bounded above by a constant. As mentioned above,
the\break 
correspondence between positive closed $(1,1)$-currents and their dynamical potentials is a\break bijection. 
Moreover, we know that quasi-p.s.h.\ functions 
(i.e., functions that are locally difference between a p.s.h.\ and a smooth function)
are integrable with respect to $\mu$ \cite[Th.\ 1.35]{dinh2010dynamics}.
Since the set of positive closed $(1,1)$-currents of
mass less than or equal to 1  
is compact, $u_S$ belongs to a compact family of p.s.h.\ functions modulo $mT$.
We deduce that there is a constant $A>0$ independent of $S$ such
that $u_S\leq A/2$ on $\P^k$, see \cite[App.\ A.2]{dinh2010dynamics} for more details.
It follows that $\sup u_S \leq \sup u_{S_0} +A/2$ because  $\sup u_{S_0}$ is non-negative.

Consider the current $S': = S_0-S$ which is positive closed and smaller
than $S_0$. By the uniqueness of the dynamical potential, we have
$u_{S'} = u_{S_0}- u_{S}$, which implies
$u_{S} = u_{S_0} - u_{S'}$. Since $S'\leq S$, as above, we also have $\sup u_{S'}\leq A/2$. It follows that 
$$\inf u_S \geq \inf u_{S_0} - \sup u_{S'} \geq \inf u_{S_0}-A/2.$$
This estimate and the above inequality $\sup u_S \leq \sup u_{S_0} +A/2$
imply the lemma.
\end{proof}

\begin{corollary}\label{c:compare-sup}
There exists a positive constant $A$ such that
for every positive closed $(1,1)$-current $S_0$  on $\P^k$
and for every  continuous function $g\colon\P^k\to \R$ 
with
$\abs{dd^c g} \leq S_0$ we have
$\Omega(g) \leq A \norm{S_0} + 3\Omega (u_{S_0})$.
\end{corollary}

\begin{proof}
By linearity
 we can assume that $S_0$ is of mass $1/2$.
Define $R:=dd^c g$ and write it as a difference of positive closed currents,
$R= (R+S_0)-S_0$. Since $R+S_0$ and $S_0$ belong to the same cohomology class, they have the same mass $1/2$.
We denote as usual by $u_{R+S_0}$ and $u_{S_0}$ the dynamical potentials
of $R+S_0$ and $S_0$ respectively. 

A direct computation gives $\ddc (g-u_{R+S_0}+u_{S_0})=0$ which implies that $g-u_{R+S_0}+u_{S_0}$ is a constant function. 
Thus,
\[
\Omega (g) =\Omega(u_{R+S_0}-u_{S_0}) \leq \Omega (u_{R+S_0}) + \Omega (u_{S_0}).
\]
The assertion follows from Lemma \ref{l:compare-sup} applied to $R+S_0, 2S_0$ instead of $S,S_0$. We use here 
the fact that $R+S_0 =\ddc g+ S_0 \leq 2S_0$ and that $2S_0$ is of mass 1.
We also use a constant $A$ which is equal to twice
the one in Lemma \ref{l:compare-sup}.
\end{proof}

The following result 
gives a quantitative
control on the oscillation of $u$ in terms of
 the oscillation of $v$. Notice in particular
that it implies that, if $v$ is H\"older or $\log^p$-continuous for some $p>0$, then
$u$ enjoys the same property with possibly a loss in the H\"older exponent, but not in the $\log^p$-exponent.

\begin{proposition}\label{p:mod-cont}
Let $u$ and $v$ be two p.s.h.\ functions on  $\B_3^k$
such that $\ddc u \leq \ddc v$ and $v$ is continuous.
Then $u$ is continuous and for every $0<s\leq 1$
there is a positive constant $A$ 
(independent of $u$ and $v$)
 such that, for every $0<r\leq 1/2$, we have 
\[m_{\B^k_1} (u, r) \leq m_{\B_{2}^k} (v, r^s) + A m_{\B_{2}^k}(u,r^s) r^{1-s}
\leq m_{\B_{2}^k} (v, r^s) + A\Omega_{\B_{2}^k} (u) r^{1-s}.
\]
\end{proposition}

\begin{proof}
The continuity of $u$ is a well-known property. Indeed, since $\ddc v-\ddc u$ is a positive
closed $(1,1)$-current, there is a p.s.h.\ function 
$u'$ such that $\ddc u'=\ddc v- \ddc u$. So, both $u+u'$ and $v$ are
potentials of $\ddc v$. We deduce that they differ by a pluriharmonic function.
Hence $u+u'$ is continuous. We then easily deduce that both $u$ and $u'$ are continuous because both
 are p.s.h.\ (and hence u.s.c.).

We prove now the estimate in the lemma.
Let $x,y \in \B^k_1$ be
such that $\|x-y\| \leq r$.  We need to bound $u(y)-u(x)$.
Without loss of generality, we 
can reduce the problem
to the case $k=1$ by restricting ourselves to the complex line through $x$ and $y$.
Moreover, by translating and adding constants
to $u$ and $v$,
we can assume that
$x=0$, $|y|\leq r$, $u(x)=v(x)=0$, 
and $u(y)\geq 0$.
It is then enough
to prove
that
\[
u(y)
\leq m_{\D_1} (v, r^s)  + 
A \Omega_{\D_{r^s}} (u) r^{1-s}
\]
for some positive constant $A$ 
and for $u,v$ defined on $\D_2$.
Note that $\Omega_{\D_{r^s}} (u)\leq 2 m_{\D_1}(u,r^s)$.

\begin{claim}
We have, for some positive constant $A$,
\[
u (y)
\leq
\frac{1}{\Leb (\partial \D_{r^s})}
\int_{|z| = r^s}   u(z) d\Leb(z)
+ 
A\Omega_{\D_{ r^s}} (u) r^{1-s}.
\]
\end{claim}

Assuming the claim, we first complete the proof of the lemma. Let $\widetilde u$ (resp. $\widetilde v$) be the
radial subharmonic function on $\D_2$ such that $\widetilde u(z)$ (resp. $\widetilde v(z)$) is
equal to the mean value of $u$ (resp. $v$) on the circle of center 0 and radius $|z|$. 
Using the Claim, in order to obtain the lemma, it is enough to show that $\widetilde u\leq \widetilde v$. 

Recall that $v-u$ is a subharmonic function vanishing at 0. Therefore, $\widetilde v-\widetilde u$ is a radial subharmonic function vanishing at 0.
Radial subharmonic functions are increasing in $|z|$. Thus,  $\widetilde v-\widetilde u$ is a non-negative function and the lemma follows.
\end{proof}

\begin{proof}[Proof of the Claim]
Define $u'(z):=u(z r^s)$ and $y':=y/r^s$.
 We need to show that, for $|y'|\leq r^{1-s}$,
\[
u' (y')
\leq 
 \frac{1}{\Leb (\partial \D_1) }
\int_{\partial\D_1}  u'(z) d\Leb(z)
+ 
A\Omega_{\D_1} (u') r^{1-s}.
\]

We can assume, without loss of generality, 
that $y'=\alpha \in \R^+$ and $\alpha\leq r^{1-s}$.
Consider the automorphism $\Psi$ of the unit disc
 given by $\Psi (z) =\frac{z+\alpha}{1+\alpha z}$.
The map $\Psi$ satisfies $\Psi(0)= y'$ and moreover $\Psi$ 
extends smoothly to $\partial \D_1$ 
and tends to the identity in the $\Cc^1$ norm
as $\alpha\to 0$. It follows that $\norm{\Psi^{\pm 1} - \id}_{\Cc^1} \leq  A'\alpha \leq A' r^{1-s}$
 for some positive constant $A'$.

Define $u'' := u' \circ \Psi$ and denote by $\nu$ the normalized standard Lebesgue measure on the unit 
circle.
We deduce from the last inequalities that $\Psi_* \nu - \nu$ is 
given by
 a smooth 1-form on $\partial\D_1$ and 
$\|\Psi_* \nu - \nu\|_\infty=O(r^{1-s})$.
Applying the submean inequality to the subharmonic function $u''$ we get
\[
u'(y') 
 = u'' (0) \leq \langle \nu, u''\rangle = \langle \nu, u'\circ \Psi \rangle=
\langle \Psi_*\nu,u'\rangle 
 =
\langle\nu , u' \rangle
+ 
\langle \Psi_* \nu - \nu,u'\rangle.
\]
Since $\Psi_* \nu$ and $\nu$ are probability measures, the integral $\langle \Psi_* \nu - \nu,u'\rangle$
does not change if we add to $u'$ a constant $c$. With the choice
 $c=- \inf_{\D_1} u'$ (observe that $u'$ is continuous on $\overline\D_1$)
 we get
\[
u'(y') \leq \int_{\partial \D_1} u' \; d\nu + \sup_{\D_1} |u'+c|\, O(r^{1-s})
\leq   \int_{\partial \D_1} u' \; d\nu
+
A
\Omega_{\D_1} (u')
r^{1-s}
\]
for some positive constant $A$.
This
implies the desired inequality.
\end{proof}

\begin{corollary} \label{c:mod-cont}
Let $v$ be a continuous p.s.h.\ function on $\B^k_3$. Let $u$ be a continuous real-valued function on $\B_3^k$ such that $|\ddc u|\leq \ddc v$.
Then for every $0<s \leq 1$ 
we have for $0<r\leq 1/2$
\[
 m_{\B_1^k} (u, r) \leq 3 m_{\B_2^k} (v, r^s)
+ 
A\pa{\Omega_{\B_2^k} (u) + \Omega_{\B_2^k} (v)} r^{1-s},
\]
where $A$ is a positive
constant independent of $u$ and $v$.
\end{corollary}

\begin{proof}
Since $|\ddc u|\leq \ddc v$, we have $\ddc(u+v)=\ddc u+\ddc v\geq 0$. So the function $u+v$ is p.s.h.; 
observe also that 
$\ddc (u+v)= \ddc u+\ddc v\leq 2\ddc v$. Therefore, we can apply Proposition \ref{p:mod-cont} to  $u+v,2v$
instead of $u,v$. 
This gives
\[
\begin{aligned}
m_{\B_1^k} (u,r) 
& \leq m_{\B_1^k} (u+v,r) + m_{\B_1^k} (v,r)
\leq
m_{\B_2^k} (2v, r^s) + A\Omega_{\B^k_2} (u+v) r^{1-s} + m_{\B_2^k} (v,r)\\
& \leq
3m_{\B_2^k} (v, r^s)+ A\pa{\Omega_{\B^k_2}(u)+ \Omega_{\B^k_2}(v)}
r^{1-s},
\end{aligned}\]
which is the desired estimate.
\end{proof}

\begin{corollary} \label{c:mod-cont-bis}
Let $S_0$ be a positive closed $(1,1)$-current on $\P^k$ with continuous local
potentials. Let $\Fc(S_0)$ denote the set of all continuous real-valued
functions $g$ on $\P^k$ such that $|\ddc g|\leq S_0$. Then 
$\Fc(S_0)$ is equicontinuous.
\end{corollary}

\begin{proof}
Let  $g$ be as in the statement.  
We cover $\P^k$ with
 a finite family of open sets of the
form $\Phi_j(\B_{1/2}^k)$ where $\Phi_j$ is an injective holomorphic map from $\B_4^k$ to $\P^k$. 
Write $S_0=\ddc v_j$ for some continuous p.s.h.\ function $v_j$ on $\Phi_j(\B_4^k)$ and
define $V_j:=\Phi_j(\B_3^k)$.  

We apply Corollary \ref{c:mod-cont} to $g$, $v_j$
restricted to $V_j$ instead of $u$, $v$ and to $s=1/2$. Taking into account the distortion
of the maps $\Phi_j$, we see that for all $r$ smaller than some constant $r_0>0$
\[
m_{\P^k}(g, r)
\leq 3 \max_j m_{V_j} (v_j, c\sqrt r) + A\Big(\Omega_{\P^k} (g) + \max_j \Omega_{V_j} (v_j)\Big) \sqrt{r},
\]
where $c\geq 1$ is a constant. Since $\Omega_{\P^k}(g)$ is bounded
by Corollary \ref{c:compare-sup}, 
the RHS of the last inequality is bounded by a constant $\epsilon_r$
which is independent of $g$ and tends to 0 when $r$ tends to 0.
It is now clear that the family $\Fc(S_0)$ is equicontinuous.
\end{proof}

\subsection{Dynamical potentials of  $(f^n)_* \omega_\FS$} \label{s:norm}
In
 this section we
consider the action of the operator
 $(f^n)_*$ on functions and currents. 
Some results and ideas here are of independent interest.
 Recall that we always assume that $f$ satisfies the Assumption {\bf (A)} in the Introduction.
 
 \medskip

\label{ss:bound-pot-push}
We start by giving estimates on the potentials
of the currents  $(f^n)_* \omega_\FS$. As explained in the Introduction,
these estimates will allow us to globally control the distortion of $f^n$.
Define
$$\omega_n := d^{-(k-1)n} (f^n)_* \omega_\FS.$$
Recall that $f_*$ multiplies the mass of a positive closed $(1,1)$-current
by $d^{k-1}$. Therefore, all currents $\omega_n$ have unit mass.
We denote by $u_n$ the dynamical potential of $\omega_n$. In particular,
 $u_0$ is the dynamical potential of $\omega_\FS$. It is known that $u_0$ is H\"older continuous,
see \cite{kosek1997holder, dinh2010dynamics}. 

Observe that $d^{-1}f^*\omega_\FS$ is a smooth positive closed $(1,1)$-form of
mass 1. Therefore, there is a unique smooth function $v$ such that
\[\ddc v= d^{-1}f^*\omega_\FS -\omega_\FS \qquad \text{and} \qquad \langle \mu,v\rangle =0.\]
 
 \begin{lemma} \label{l:u0}
 We have 
 $$u_n=d^{-(k-1)n}(f^n)_*u_0 \qquad \text{and}\qquad u_0= - \sum_{n=0}^\infty d^{-n} v\circ f^n.$$
 \end{lemma}
 \proof
 We prove the first identity. Denote by $u_n'$ the RHS of this identity, 
 which is a continuous function. By the
 definition of $u_n$ and the invariance of $T$, we have 
 \[\ddc (u_n-u_n') = (\omega_n-T) - d^{-(k-1)n}(f^n)_*(\omega_\FS-T)=  (\omega_n-T) - (\omega_n-T) =0.\]
 Therefore, $u_n-u_n'$ is pluriharmonic and hence constant on $\P^k$.
 Moreover, the invariance of $\mu$ implies that
 \[\langle \mu,u_n'\rangle =  d^{-(k-1)n}\langle (f^n)^*\mu, u_0\rangle = d^n\langle \mu, u_0\rangle =0.\]
 By the definition of $u_n$, we also have $\langle \mu, u_n\rangle =0$.
 We deduce that $u_n=u_n'$, 
 which implies the first identity in the lemma.
 
 It is clear that the sum in the RHS of the second identity in the lemma converges uniformly. 
 Therefore, this RHS is a continuous function that we denote by $u_0'$. 
 The invariance of $\mu$ also implies that $\langle \mu,u_0'\rangle=0$.
 A direct computation gives
 \[\ddc u_0'=\lim_{N\to\infty}  \Big( - \sum_{n=0}^{N-1} d^{-n} \ddc (v\circ f^n)\Big)=\lim_{N\to\infty}\omega_\FS-
 d^{-N} (f^N)^*\omega_\FS=\omega_\FS-T,\]
 where the last identity is a consequence of the definition of $T$. 
 Since $\ddc u_0$ is also equal to $\omega_\FS-T$, we obtain that $u_0-u_0'$
 is constant on $\P^k$. Finally, using that 
 \[\langle \mu,u_0\rangle = \langle \mu, u_0'\rangle =0,\]
 we conclude that $u_0=u_0'$. This ends the proof of the lemma.
 \endproof

In the sequel, we will need explicit bounds
on the oscillation $\Omega (u_n)$ of $u_n$.
These are provided in the next result.

\begin{lemma}\label{l:u-n-sup}
For every constant $A>1$, there exists a positive constant $c$ independent of  $n$ such that  
$\norm{u_n}_{\infty} \leq c A^n$ and  $\Omega (u_n) \leq c A^n$ for all $n\geq 0$.
\end{lemma}

\begin{proof}
Observe that the second assertion is deduced from 
the first one by
 replacing $c$ with $2c$. We prove now the first assertion. 
By Lemma \ref{l:u0} we have, for any given $z\in \P^k$,
\[
\begin{aligned}
u_n (z) & = d^{-(k-1)n} \pa{\pa{f^n}_* u_0} (z)
 = \big\langle \delta_z,  d^{-(k-1)n} \pa{f^n}_* u_0 \big\rangle\\
& = d^n \big\langle d^{-kn} \pa{f^n}^* \delta_z , u_0 \big\rangle
 = d^n \Big\langle d^{-kn} \pa{f^n}^* \delta_z ,  - \sum_{m=0}^\infty d^{-m} v\circ f^m \Big\rangle \\
 & = -d^n  \Big\langle d^{-kn} \pa{f^n}^* \delta_z ,  \sum_{m=0}^n d^{-m} v\circ f^m \Big\rangle 
 -  \Big\langle d^{-kn} \pa{f^n}^* \delta_z ,   \sum_{m=n+1}^\infty d^{-m+n} v\circ f^m \Big\rangle.
\end{aligned}
\]
The absolute value of the second term in the last line  is bounded by $\norm{v}_{\infty}$
because $d^{-kn} \pa{f^n}^* \delta_z$ is a probability measure.
Observe that
$(f^n)_* (v \circ f^m) = d^{km} (f^{n-m})_* v$ for all $n\geq m$.
Hence,
the absolute value of the first term is equal to 
\begin{equation}\label{e:sum-v}
\Big|\sum_{m=0}^n   d^{n-m} \Big\langle
 \delta_z ,  d^{-k(n-m)} (f^{n-m})_*v \Big\rangle\Big|
\leq \sum_{j=0}^n d^j \|d^{-kj}(f^j)_*v\|_\infty.
\end{equation}
Under the Assumption {\bf (A)},
 it is known that $\|d^{-kj}(f^j)_*v\|_\infty \lesssim \delta^{-j}$ for every $0<\delta<d$. Indeed, the Assumption
 {\bf (A)} 
 implies the property {\bf (A1)}
  below, see \cite[Cor. 1.2]{dinh2010equidistribution}.

\medskip\noindent
{\bf (A1)} Let $g\colon \P^k \to \R$ be $\Cc^2$ and such that $\langle\mu,g\rangle=0$.
For every constant $1<\delta<d$, there is a positive constant $c$
 independent of $g$ and $n$ such that
$$\|d^{-kn}(f^n)_*g\|_\infty \leq c\|g\|_{\Cc^2} \delta^{-n}.$$

\medskip

By choosing $\delta>d/A$, we can bound the RHS of \eqref{e:sum-v} 
by a constant times $A^n$. This ends the proof of the lemma.
\end{proof}

As an application of the previous estimates, we have the following lemma that can be used to study the regularity of
functions $g\colon \P^k\to \R$.

\begin{lemma}\label{l:ddc-dyn-log}
Let $g\colon \P^k \to \R $ be a continuous function
and $0<\beta<1$ a 
constant such that 
\begin{equation}\label{e:ddcg-omegan}
\abs{dd^c g} \leq \sum_{n=0}^\infty  \beta^n \omega_n.
\end{equation}
Then, for every $q>0$, there is a positive constant $c=c(q, \beta)$ independent of $g$ such that
$$\norm{g}_{\log^q} \leq c.$$
\end{lemma}

\begin{proof}
We 
bound the continuity modulus $m(g,r)$ of $g$ by means of
Corollary \ref{c:mod-cont}. We only need to consider $0<r\leq 1/2$. 
For this purpose, since $T$ has H\"older continuous local potentials, it suffices to
bound the continuity modulus of the dynamical potential of the RHS of \eqref{e:ddcg-omegan}.
This dynamical potential is equal to
\[
u:=\sum_{n=0}^{\infty} \beta^n u_n.
\]

Fix a constant $1 < A < 1/\beta$. By Lemma \ref{l:u-n-sup}, we have
$\norm{u_n}_{\infty}\lesssim A^n$. Hence, for every $N$, we have
\[
\begin{aligned}
m(u,r)&\lesssim
\sum_{n\leq N} \beta^n  m ( u_n, r) + \sum_{n>N} (A\beta)^n \lesssim \sum_{n\leq N} \beta^n  m ( u_n, r) + (A\beta)^N.
\end{aligned}
\]
Applying \cite[Cor.\,4.4]{dinh2010equidistribution} inductively to some iterate of $f$,  we see
that the Assumption {\bf (A)} implies:

\medskip\noindent
{\bf (A2)} for every constant $\kappa>1$, there 
are an integer $n_\kappa \geq 0$ and 
a constant $c_\kappa>0$
independent of $n$
 such that for all $x,y\in\P^k$ and $n\geq n_\kappa$ 
 we can write
$f^{-n}(x)=\{x_1,\ldots, x_{d^{kn}}\}$ and $f^{-n}(y)=\{y_1,\ldots, y_{d^{kn}}\}$
(counting multiplicity) with the property 
that
$$\dist(x_j,y_j)\leq c_\kappa \dist(x,y)^{1/\kappa^n} \quad \text{for } \ j=1,\ldots, d^{kn}.$$

\medskip

 By definition, the function $u_0$ is $\gamma$-H\"older continuous for
 some H\"older exponent $\gamma$ because $T$ has H\"older continuous local potentials. 
The above property {\bf (A2)}
 implies that 
$(f^n)_* u_0$ is $\gamma \kappa^{-n}$-H\"older continuous
for all $n\geq n_\kappa$. 
More precisely, we have
\[
m(d^{-kn}(f^n)_* u_0,r)\leq c' r^{\gamma \kappa^{-n}} 
\qquad
 \text{and hence} \qquad
m(u_n, r)\leq c'd^n r^{\gamma \kappa^{-n}}
\]
for some positive constant $c'$
 independent of $n\geq n_\kappa$ 
and $r$.
Observe also that
for $0 \leq n \leq n_\kappa$ all the $u_n$ are
$\alpha_\kappa$-H\"older
continuous for some $\alpha_\kappa >0$. Indeed, as the multiplicity
of $f^n$ at a point is at most $d^{kn}$, we have (see again \cite[Cor.\,4.4]{dinh2010equidistribution}):

\medskip\noindent
{\bf (A2')}
there is a constant $c_0>0$ such that for every $n\geq 0$,
for all $x,y\in\P^k$,
 we can write
$f^{-n}(x)=\{x_1,\ldots, x_{d^{kn}}\}$ and $f^{-n}(y)=\{y_1,\ldots, y_{d^{kn}}\}$
(counting multiplicity) with the property 
that
$$\dist(x_j,y_j)\leq c_0 \dist(x,y)^{1/d^{kn}} \quad \text{for } \ j=1,\ldots, d^{kn}.$$

Therefore, we have
\begin{equation}\label{e:cont-mod-u}
m(u,r)
\lesssim r^{\alpha_\kappa} 
+ \sum_{n_\kappa \leq n\leq N} \pa{\beta d}^n r^{\gamma \kappa^{-n}} + (A\beta)^N.
\end{equation}
Choose $\kappa$ close enough to 1 so that $2q\log\kappa < |\log(A\beta)|$ and take
$$N={1\over 2\log\kappa}  \log \abs{\log r}$$
(recall that we only need to consider $r\leq 1/2$). Then, the last term in \eqref{e:cont-mod-u} satisfies 
$$(A\beta)^N= e^{N\log(A\beta)} < e^{-2Nq\log\kappa}=|\log r|^{-q}.$$
It remains to prove that the 
sum in \eqref{e:cont-mod-u} satisfies a similar estimate. We have
\[ \sum_{n\leq N} \pa{\beta d}^n r^{\gamma \kappa^{-n}}
 \leq \sum_{n\leq N} \beta^n  d^N r^{\gamma \kappa^{-N}}
 \lesssim  d^N r^{\gamma \kappa^{-N}}
=e^{{\log d\over  2\log\kappa}  \log \abs{\log r}} e^{\gamma (\log r) e^{-{1\over 2}\log|\log r|}}
={|\log r|^{\log d\over  2\log\kappa}   \over e^{\gamma \sqrt{|\log r|}}} \cdot\]

The last expression is smaller than a constant times $|\log r|^{-q}$
because $e^t \gg t^M$ when $t\to\infty$ for every $M\geq 0$. This,
together with the above estimates, gives
$m(u,r)\lesssim |\log r|^{-q}$ and ends the proof of the lemma.
\end{proof}

\section{Existence of the scaling ratio and equilibrium state}\label{s:lambda-rho}

 In this section
we prove the existence of a
good scaling ratio $\lambda$,
see Theorem \ref{t:main-lam} below.

\subsection{Main statement and first step of the proof}\label{ss:lambda-rho-first-step}

Recall that the Perron-Frobenius
operator $\Ll$ is defined 
as
in \eqref{e:L}. A direct computation gives 
 \[ \Ll^n (g) (y)
 = \sum_{f^n (x) =y } 
e^{ \phi (x) + \phi (f(x) ) + \cdots + \phi (f^{n-1} (x)) }  g(x).
 \]

\begin{theorem}\label{t:main-lam}
Let $f$ and $\phi$ be as in Theorem \ref{t:main}.
There exist a number $\lam>0$ and a continuous function
$\rho>0$ on $\P^k$ such that
for every continuous function $g\colon \P^k \to \R$
the sequence
$\lam^{-n}\Ll^n(g)$ is equicontinuous and converges uniformly to $c_g\rho$, where $c_g$
is a constant depending linearly on $g$. Moreover, 
if $g$ is strictly positive, then $c_g$ is strictly positive and the sequence $\Ll^n(g)^{1/n}$
converges uniformly to $\lambda$ as $n$ tends to infinity.
\end{theorem}

We will first study the case where $g$ is equal to $\1$. The
 general case will be deduced from this particular case. Define
$\1_n := \Ll^n (\1)$. Denote by $\rho_n^+$ and $\rho_n^-$
the maximum and the minimum of $\1_n$, respectively.
Consider also the ratio $\theta_n := \rho^+_n / \rho^-_n$ and the function 
$\1_n^*:=(\rho_n^-)^{-1}\1_n$. Observe that the last function satisfies $\min \1_n^*=1$. 
The following result will be crucial for us.

\begin{proposition}\label{p:theta-n-1-n}
Under the hypotheses of Theorem \ref{t:main-lam}, the sequence
$\set{\theta_n}$ is bounded and the sequence of functions
$\set{\1_n^*}$ is uniformly bounded and equicontinuous.
\end{proposition}

The proof of this result will 
 be given in Section \ref{ss:proof-theta-n-1-n} and uses
 the technical tools that were
 presented in Section \ref{s:preliminary}. 
 Before giving it, we 
  need to first introduce some auxiliary objects.

By Lemma \ref{l:g-Cs-logp} applied to $\phi$ instead of $g$, we can find functions $\phi_n$ and $\psi_n$ such that 
\begin{equation} \label{e:phi-n}
\phi = \phi_n + \psi_n,\qquad \norm{\phi_n}_{\Cc^2} \leq  c \norm{\phi}_{\infty}  e^{{1\over 2}n^{2/q}},
 \qquad \text{and} \qquad  \norm{\psi_n}_{\infty} \leq c  \norm{\phi}_{\log^q} n^{-2}.
\end{equation}
Consider two integers $J\geq 0$ and  $N\geq 0$, 
whose values will be specialised later. Define for $n\geq N+1$
 \begin{equation}\label{e:l-hat}
  {\hat \Ll_n} (g)(x)
 := \sum_{f^n (x) =y } 
e^{ \phi_{n+J} (x) + \phi_{n+J-1} (f(x) ) + \cdots + \phi_{J+N+1} (f^{n-N-1} (x)) }  g(x).
 \end{equation}
This operator will be used to approximate
 $\Ll^n$. The gain here is the fact that the involved functions $\phi_m$ have controlled $\Cc^2$ norms. 
As above, we define
$$\hat \1_n:=\hat\Ll_n \1,\quad \hat \rho^+_n:=\max \hat\1_n, \quad \hat \rho^-_n
:=\min \hat\1_n, \quad \hat\theta_n:=\hat\rho^+_n/\hat\rho^-_n, \quad \mbox{ and }
\quad \hat\1_n^*:=(\hat\rho^-_n)^{-1} \hat\1_n.$$
The following lemma allows us to reduce our problem to the study of the functions $\hat \1_n$.

\begin{lemma}\label{l:hat-theta-n}
There exists a positive constant $c=c(N)$
such that, for all $n>N\geq 0$ and $J$,
\[
c^{-1} \leq 
\rho^+_{n}/{\hat \rho^+_n}\leq c
\quad \mbox{ and }
\quad 
c^{-1} \leq \rho^-_{n}/{\hat \rho^-_n}\leq c.\]
In particular,
the sequence $\big\{\hat \theta_n \big\}$ is bounded if and only if the sequence
 $\set{\theta_n}$ is  bounded.
\end{lemma}

\begin{proof}
We have 
\[\begin{aligned}
\rho^+_{n}& = \max\1_{n} = \max_{y} \sum_{f^{n}(x)=y}e^{ \phi (x) + \phi (f(x) ) + \cdots + \phi (f^{n-1} (x)) } \\
& =
\max_{y}
\sum_{f^{n}(x) =y }
e^{ \phi_{n+J} (x) 
 + \cdots
 + \phi_{J+N+1} (f^{n-N-1} (x)) } 
 \cdot
 e^{\psi_{n+J} (x)
  + \cdots
   + \psi_{J+N+1} (f^{n-N-1} (x)) }\\
  &\qquad \qquad \qquad  \qquad \cdot
   e^{\phi(f^{n-N} x)) + \cdots + \phi(f^{n-1}(x))}
\end{aligned}
\]
and similarly for $\rho^-_{n}$.
So, both
$\rho^+_{n}/{\hat \rho^+_n}$ and $\rho^-_{n}/{\hat \rho^-_n}$ are bounded from above and below
by $e^{N \max \phi} C_{n,N,J}$ and $e^{N\min \phi}/C_{n,N,J}$ 
respectively, where
\[
C_{n,N,J} :=
 e^{\norm{\psi_{n+J}}_\infty + \norm{\psi_{n+J-1}}_\infty + \, \cdots \, +  \norm{\psi_{J+N+1}}_\infty }.
\]
It follows from the estimate on $\psi_n$ given above that 
$C_{n,N,J}$
 is bounded from above 
  by a positive constant
  which does not depend on $n,J$ and $N$. 
  Therefore, 
  both $\rho^+_{n}/{\hat \rho^+_n}$ and $\rho^-_{n}/{\hat \rho^-_n}$ 
  are bounded from below and above 
 by positive constants as in the statement. The lemma follows.
\end{proof}

\subsection{An estimate for  $dd^c \hat{\1}_n$}\label{ss:estimate-ddc1n}

Proposition \ref{p:theta-n-1-n} will be obtained using the following crucial estimate for 
$dd^c \hat \1_n$.
We will see here the role of the estimate of 
the $\Cc^2$ norm of $\phi_n$.
Recall that $q>2$, see Theorem \ref{t:main}.
We also refer to Section \ref{ss:bound-pot-push} for notation.

\begin{proposition} \label{p:laplace-1-n}
There exists a sub-exponential function $\eta(t)=ct^3e^{(t+J)^{2/q}}$ 
with a positive constant $c=c(\norm{\phi}_{\log^q}, \norm{\phi}_\infty)$
 independent of $n,J$ and $N$  such that
for all $n>N \geq 0$  we have
$$\abs{dd^c \hat \1_n} \leq  \sum_{m=N+1}^{n-N}
\eta(m) e^{m\max \phi} \hat \rho^+_{n-m} d^{(k-1)m} \omega_{m} + 
\sum_{m=n_0}^{n} d^{kN} \eta(m) e^{(n-N)\max \phi}d^{(k-1)m} \omega_{m},
$$
where $n_0 :=\max(n-N+1,N +1)$.
\end{proposition}

Recall that the function $\hat \1_n$ 
is given by
\[
\hat \1_n (y) = \sum_{f^{n}(x)=y} e^{\phi_{n+J} (x) + \phi_{n+J-1} (f(x)) + \cdots + \phi_{J+N+1} (f^{n-N-1} (x)) }.
\]
In order to estimate $dd^c \hat \1_n$, we will use a now classical
construction due to Gromov \cite{gromov2003entropy}.
Define the manifold $\Gamma_n \subset (\P^k)^{n+1}$ by
\[
\Gamma_n :=  \big\{(x,f(x), \dots , f^{n} (x)) : x \in \P^k \big\},
\]
which can also be seen as the graph of the map $(f,f^2, \dots , f^n)$
in the product space $(\P^k)^{n+1}$.
Consider the function $\h$
on $(\P^k)^{n+1}$ given by
\[
\h(x_0, \dots, x_n):= e^{\phi_{n+J} (x_0)+\phi_{n+J-1}(x_1) + \dots + \phi_{J+N+1} (x_{n-N-1})}.
\]
The function $\hat \1_n$ on $\P^k$ is equal to the push-forward of the
function $\h_{|\Gamma_n}$ to the last factor $\P^k$ of $(\P^k)^{n+1}$.
Indeed, denoting by $\pi_n$ the restriction of the projection
$x\mapsto x_n$ to $\Gamma_n$, we have
\[
\pa{\pi_n}_*(\h) (y) =
\sum_{(x_0, \dots, x_n)\in \Gamma_n \colon x_n =y} \h(x)
=\sum_{x\in f^{-n} (y)} e^{\phi_{n+J} (x) + \cdots + \phi_{J+N+1} (f^{n-N-1} (x)) }
=\hat \1_n(y).\]

Recall that, since $\ddc \hat \1_n$ is real, 
estimating $|dd^c \hat \1_n|$
 means finding a good positive closed $(1,1)$-current $S$ on $\P^k$
such that both $S\pm dd^c \hat \1_n$ are positive.
According to the identities above, we have
\[
{dd^c \hat \1_n} =
{\pa{\pi_n}_* \big(dd^c \h} \big).
\]
Thus, we need to estimate $dd^c \h$ on $(\P^k)^{n+1}$ and $\Gamma_n$.
We define  $\omega^{(m)}$ as 
the pullback of the Fubini-Study form $\omega_\FS$ to $(\P^k)^{n+1}$  
by the projection $x\mapsto x_m$.
Equivalently, $\omega^{(m)}$ is a $(1,1)$-form on $(\P^k)^{n+1}$ such that 
$\omega^{(m)}(x)=\omega_\FS(x_m)$.

\begin{lemma} \label{l:ddch} 
There exists a sub-exponential function $\eta(t)=ct^3e^{(t+J)^{2/q}}$ with a positive
constant $c=c(\norm{\phi}_{\log^q},\norm{\phi}_{\infty})$ independent of $n,J$, and $N$  such that
$$\abs{dd^c \h} \leq  \h \sum_{m=0}^{n-N-1} \eta(n-m)  \omega^{(m)}.$$
\end{lemma}

\begin{proof}
A direct computation gives
\begin{equation*} 
i\partial \bar{\partial} \h
= \h  \Big( \sum_{m=0}^{n-N-1} i\ddbar \phi_{n+J-m}(x_m)
+\sum_{m,m'=0}^{n-N-1} i\partial \phi_{n+J-m}(x_m) \wedge \dbar \phi_{n+J-m'}(x_{m'}) \Big).
\end{equation*}
For the first sum, observe that 
$$|i\ddbar \phi_{n+J-m}(x_m)|\lesssim \|\phi_{n+J-m}\|_{\Cc^2} \omega^{(m)} (x)\lesssim e^{(n+J-m)^{2/q}} \omega^{(m)}(x).$$
For the second sum, consider $m'\leq m \leq n-N-1$. By using Cauchy-Schwarz's inequality, we have
\begin{eqnarray}\label{e:phi_cs_1}
\lefteqn{ | i\partial   \phi_{n+J-m}(x_m) \wedge \dbar \phi_{n+J-m'}(x_{m'})| }\\
 & \leq & (m-m'+1)^{-2}  i\partial \phi_{n+J-m}(x_m) \wedge \dbar \phi_{n+J-m}(x_m) \nonumber \\
  &  &  + (m-m'+1)^2  i\partial \phi_{n+J-m'}(x_{m'}) \wedge \dbar \phi_{n+J-m'}(x_{m'}) \nonumber \\
& \lesssim & (m-m'+1)^{-2} \|\phi_{n+J-m}\|_{\Cc^1}^2\omega^{(m)}(x) + (m-m'+1)^2 \|\phi_{n+J-m'}\|_{\Cc^1}^2\omega^{(m')}(x) \nonumber \\
& \lesssim & (m-m'+1)^{-2}  e^{(n+J-m)^{2/q}} \omega^{(m)}(x) + (n-m'+1)^2 e^{(n+J-m')^{2/q}} \omega^{(m')}(x). \nonumber
\end{eqnarray}
This
and the fact that $\sum_{j=1}^\infty j^{-2}$ is finite
imply
 that
\begin{eqnarray} \label{e:phi_cs_2} 
\lefteqn{\Big|  \sum_{0\leq m'\leq m\leq n-N-1}  i\partial \phi_{n+J-m}(x_m) \wedge \dbar \phi_{n+J-m'}(x_{m'})\Big|}  \\
&\lesssim & \sum_{m=0}^{n-N-1} e^{(n+J-m)^{2/q}} \omega^{(m)}(x) + \sum_{m'=0}^{n-N-1} (n-m'+1)^3 e^{(n+J-m')^{2/q}} \omega^{(m' )}(x) \nonumber \\
&\lesssim &
\sum_{m=0}^{n-N-1} (n-m)^3e^{(n+J-m)^{2/q}} \omega^{(m)}(x). \nonumber
\end{eqnarray}
We obtain by symmetry a similar estimate for the case where $m< m'\leq n-N-1$. 

Finally, combining all the above identities and estimates we get
$$|i\partial \bar{\partial} \h| \lesssim \h \sum_{m=0}^{n-N-1} (n-m)^3e^{(n+J-m)^{2/q}} \omega^{(m)}.$$
The lemma follows. 
\end{proof}

\begin{proof}[Proof of Proposition \ref{p:laplace-1-n}]
We are only
interested in the restriction of $\h$ to the graph $\Gamma_n$.
We deduce from Lemma \ref{l:ddch} that
\begin{equation} \label{e:sum-ddc-0}
\abs{dd^c \hat \1_n} = \abs{\pa{\pi_n}_* dd^c \h}
\leq  \sum_{m=0}^{n-N-1} \eta(n-m) { \pa{\pi_n}_*  \big(\h  \omega^{(m)} \big)}.
\end{equation}

We split the last sum into the two sums corresponding to $m<N$ and $m\geq N$. 
Note that when $n\leq 2N$, in the sum
in \eqref{e:sum-ddc-0} 
we always have $m<N$ and the first sum in the statement of the proposition 
vanishes. 
So, for simplicity, we assume that $n>2N$ and we will see 
 in the proof below that the arguments also work
when $n\leq 2N$.

For $m< N$, using the definition of $\phi_m$ 
we have $\|\phi-\phi_m\|_\infty=\|\psi_m\|_\infty\leq c'm^{-2}$ and
hence $\max\phi_m\leq \max\phi+c'm^{-2}$ for
some positive
constant $c'$ 
which may depend on
$\norm{\phi}_{\log^q}$. 
It follows that $\h\lesssim e^{(n-N)\max\phi}$. 
Then, using the definition of $\Gamma_n$, we have for $m< N$
\begin{eqnarray*}
\pa{\pi_n}_* \big( \h  \omega^{(m)}\big) &\lesssim& e^{(n-N)\max\phi} \pa{\pi_n}_* \big(\omega^{(m)}\big)
=e^{(n-N)\max\phi}
{d^{km}} (f^{n-m})_*(\omega_\FS)\\
&=& e^{(n-N)\max\phi} {d^{km}} d^{(k-1)(n-m)}\omega_{n-m}.
\end{eqnarray*}
Thus, 
\begin{eqnarray} \label{e:sum-ddc-1}
\sum_{m=0}^{N-1} \eta(n-m) { \pa{\pi_n}_*  \big(\h  \omega^{(m)} \big)} &\lesssim& \sum_{m=0}^{N-1} \eta(n-m)  e^{(n-N)\max\phi} 
{d^{km}}
d^{(k-1)(n-m)}\omega_{n-m} \nonumber \\
&\leq &  \sum_{m=n-N+1}^n {d^{kN}}
 \eta(m)  e^{(n-N)\max\phi} d^{(k-1)m}\omega_m.
\end{eqnarray}
The last expression is the second sum in the statement
of
the present proposition
(this step also works for $n\leq 2N$ 
but in this case the above sums $\sum_{0}^{N-1}$ and 
$\sum_{n-N+1}^n$ are replaced by $\sum_{0}^{n-N-1}$ and 
$\sum_{N +1}^n$ respectively). 
In order to finish the proof, it is enough to have a similar estimate for $m\geq N$
(this step is superfluous when $n\leq 2N$, see \eqref{e:sum-ddc-0}).

As above, using the definition of $\h$ and the estimates
on $\max\phi_m$ and $\|\phi-\phi_m\|_\infty$,
we have 
$$\h\lesssim e^{(n-m)\max\phi} e^{\phi(x_0)+\phi(x_1)+\cdots+\phi(x_{m-N-1})}
\lesssim e^{(n-m)\max\phi} \h'$$
with  
$$\h':=e^{\phi_{m+J}(x_0)+\phi_{m+J-1}(x_1)+\cdots+\phi_{J+N+1}(x_{m-N-1})}.$$
Note that 
the sum in the definition of $\h'$ contains $m-N$ terms while the one of $\h$ contains $n-N$ terms.
The specific choice of $\h'$ is convenient for our next computation as it is related to the function $\hat\1_m$.

Consider the map $\pi':\Gamma_n\to(\P^k)^{n-m+1}$ defined by $\pi'(x):=x':=(x_m,\ldots,x_n)$.
Denote by $\Gamma'$ the image of $\Gamma_n$ by $\pi'$.
It is the graph of the map $(f,\ldots, f^{n-m})$ from $\P^k$ to $(\P^k)^{n-m}$.  
We also have for $x'\in\Gamma'$
$$\pi'^{-1}(x')=\big\{ \big(y,f(y),\ldots,f^{m-1}(y),x'\big) \quad \text{with} \quad y\in f^{-m}(x_{m})\big\}.$$
So $\pi':\Gamma_n\to\Gamma'$ is a ramified covering of degree $d^{km}$.

Consider  the map $\pi'':\Gamma'\to\P^k$ defined by $\pi''(x'):=x_n$. We have, for $x_n\in\P^k$,
$$\pi''^{-1}(x_n)=\big\{\big(z,f(z),\ldots,f^{n-m}(z)\big) \quad \text{with} \quad z\in f^{-n+m}(x_n)\big\}.$$
So $\pi'':\Gamma'\to\P^k$ is a ramified covering of degree $d^{k(n-m)}$.
We have
  $\pi_n=\pi''\circ\pi'$. Observe that $\pi'_*(\h'\omega^{(m)})$ is a $(1,1)$-form on $\Gamma'$ such that 
\begin{eqnarray*}
\pi'_*(\h'\omega^{(m)})(x') &=& 
\Big(\sum_{y\in f^{-m}(x_{m})} e^{\phi_{m+J}(y)+\cdots+\phi_{J+N+1}(f^{m-N-1}(y))} \Big) \omega_\FS(x_m) \\
&\leq& \hat\rho^+_m \omega_\FS(x_m)=: \hat\rho^+_m\omega'(x'),
\end{eqnarray*}
where we define $\omega'$ as the pull-back of $\omega_\FS$ to $\Gamma'$
by the map $x'\mapsto x_m$. We also have
$$\pi''_*(\omega')(x_n) = \sum_{x_m\in f^{-n+m}(x_n)} \omega_\FS(x_m)
=  (f^{n-m})_*(\omega_\FS)(x_n)= d^{(k-1)(n-m)}\omega_{n-m}(x_n).$$
Thus,
$$(\pi_n)_*(\h\omega^{(m)})\lesssim  e^{(n-m)\max\phi}\pi''_*\pi'_*(\h'\omega^{(m)})
\leq  e^{(n-m)\max\phi}\hat\rho_{m}^+  d^{(k-1)(n-m)}\omega_{n-m}$$
and
\begin{eqnarray} \label{e:sum-ddc-2}
\sum_{m=N}^{n-N-1} \eta(n-m) { \pa{\pi_n}_*  \big(\h  \omega^{(m)} \big)} &\lesssim& \sum_{m=N}^{n-N-1} \eta(n-m)  
e^{(n-m)\max\phi} \hat\rho^+_{m}  d^{(k-1)(n-m)}  \omega_{n-m} \nonumber \\
&=&  \sum_{m=N+1}^{n-N} \eta(m)  e^{m\max\phi} \hat\rho^+_{n-m}  d^{(k-1)m}\omega_m. 
\end{eqnarray}
Finally, we deduce the proposition
 from \eqref{e:sum-ddc-0},
\eqref{e:sum-ddc-1}, and \eqref{e:sum-ddc-2}
by multiplying $\eta$ with a large enough constant.
\end{proof}

\subsection{Proof of Proposition \ref{p:theta-n-1-n}} \label{ss:proof-theta-n-1-n}
We are working under the hypotheses of Theorem \ref{t:main-lam}.
We will obtain Proposition \ref{p:theta-n-1-n}  using Lemmas \ref{l:theta-bound} and \ref{l:eq-1-n} below.

\begin{lemma}\label{l:theta-bound}
Under the hypotheses of Theorem \ref{t:main-lam}, given an integer $J\geq 0$, we have 
$\hat\theta_n\leq d^{kN}$
  for all $n>N$, with $N$ large enough. In particular, the sequences $(\theta_n)$ 
  and $(\hat\theta_n)$ are bounded for all $J\geq 0$ and $N\geq 0$.
\end{lemma}

\begin{proof}
Observe that the last assertion is a consequence of the first one. Indeed,
we can first fix $J$ and $N$ satisfying the first assertion of the lemma. Then, by Lemma \ref{l:hat-theta-n}, the 
sequence $(\theta_n)$ is bounded. Applying again
Lemma \ref{l:hat-theta-n}
 for arbitrary $J$ and $N$ gives that the sequence 
$(\hat\theta_n)$ is also bounded. We prove now the first assertion in the lemma with $J$ fixed and $N$ large enough.

Observe that, by the definition of $\hat\rho^\pm_n, \hat \theta_n$, and $\Omega(\cdot)$, for every $K\geq 1$
 the two inequalities 
$\hat \theta_n \leq K$ and $\Omega (\hat \1_n) \leq (K-1)\hat \rho^-_n$
are equivalent.
Hence, in order to 
get the first assertion in the lemma, it is enough to show that $\Omega (\hat \1_n) / \hat \rho^-_n \leq d^{kN}/2$.
The constants that we use below are independent of $N$ and $n$.
Fix a constant $\delta$ such that $e^{\Omega (\phi)}< \delta<d$.
By 
the estimate on
 $\|\psi_n\|_\infty$ in
\eqref{e:phi-n}, for every $j$ sufficiently large, we have $\Omega (\phi_j)\leq \Omega (\phi) + \Omega (\psi_j) < \log \delta$.
Since we assume that $N$ is large enough, the last inequality holds for  all 
$j\geq N$.

We use Proposition \ref{p:laplace-1-n} and Corollary \ref{c:compare-sup}
in order to estimate $\Omega(\hat \1_n)$ in terms of $\Omega(u_m)$.
Recall that $u_m$ is the dynamical potential of $\omega_m$.
We also use Lemma \ref{l:u-n-sup}, which gives 
 $\Omega(u_m)\lesssim d^m \delta'^{-m}$
for any $\delta'$ such that $\delta <\delta'<d$.
 More precisely, we obtain from those results
that
$$ \Omega (\hat \1_n )  \lesssim  
\sum_{m=N+1}^{n-N} \eta(m) 
e^{m  \max \phi} {d^{km}}\delta'^{-m}
\hat \rho^+_{n-m}
+
\sum_{m=\max (n-N+1, N +1)}^{n}
d^{kN} \eta(m) e^{(n-N)\max \phi}
 d^{km}
 \delta'^{-m}.$$
Since 
$\delta<\delta'$ and
 $N$ is large, the fact that $\eta$ is sub-exponential and independent of $n$ and $N$ implies that
\begin{equation} \label{e:Omega-hat-1}
 \Omega (\hat \1_n )\lesssim 
\sum_{m=N+1}^{n-N} e^{m \max \phi} d^{km} \delta^{-m}\hat \rho^+_{n-m}
 +\sum_{m=\max (n-N+1, N+1)}^{n}
d^{kN}
e^{(n-N)\max \phi}
d^{km}\delta^{-m}.
\end{equation}

We now distinguish two cases.

\medskip\noindent
{\bf Case 1.} Assume that $N<n\leq 2N$. In this case, the first sum
in \eqref{e:Omega-hat-1} is empty. We thus deduce from \eqref{e:Omega-hat-1} that
$$\Omega (\hat \1_n )\lesssim
 d^{kN} {e^{ (n-N)\max \phi}}
\Big(\frac{ d^{k(N+1)}}{\delta^{N+1}} + \cdots +\frac{ d^{kn}}{\delta^{n}}\Big)
\lesssim d^{kN} {e^{ (n-N)\max \phi}} \frac{ d^{kn}}{\delta^{n}}\cdot$$
On the other hand, by
the
definitions of
 $\hat\rho^-_n$ we have the following general estimates 
(with $n\geq N$ in the first inequality and $n-m\geq N$ in the second one)
\begin{equation} \label{e:hat-rho-n-lower}
\hat \rho^-_{n}
 \gtrsim
d^{kn}
e^{(n-N)\min \phi}
\qquad \text{and} \qquad \hat\rho_n^-\gtrsim d^{km} e^{m\min\phi} \hat\rho^-_{n-m}.
\end{equation}
The first inequality and the above estimate of $\Omega (\hat \1_n)$ imply that
$$ {\Omega (\hat \1_n )\over \hat \rho^-_{n}} \lesssim
 d^{kN} {e^{ (n-N)\Omega(\phi)} \over \delta^n} \leq  d^{kN} {e^{n\Omega(\phi)} \over \delta^n}\cdot$$
Hence, $\Omega (\hat \1_n )/\hat \rho^-_{n} \leq d^{kN}/2$
 because $N$ is chosen large enough and $\delta>e^{\Omega(\phi)}$. The lemma in this case follows.

\medskip\noindent
{\bf Case 2.} Assume now that $n>2N$. By induction on $n$ and the previous case, we can
assume that  
$\Omega (\hat \1_m )/\hat \rho^-_{m} \leq d^{kN}/2$, which
 implies
 $\hat\rho_m^+\leq d^{kN} \hat\rho_m^-$, for all $m<n$. We need to
 prove the same inequality for $m=n$. From \eqref{e:Omega-hat-1} and the induction hypothesis, we have 
\begin{eqnarray*}
\Omega (\hat \1_n ) &\lesssim& 
d^{kN}
\sum_{m=N+1}^{n-N} e^{m \max \phi} d^{km}  \delta^{-m} \hat \rho^-_{n-m} +
d^{kN}\sum_{m=n-N+1}^{n}
e^{(n-N)\max \phi} d^{km}\delta^{-m} \\
&\lesssim &  d^{kN}  \sum_{m=N+1}^{n-N} e^{m \max \phi} d^{km} \delta^{-m} \hat \rho^-_{n-m} 
+
d^{kN}
e^{(n-N)\max \phi} d^{kn}\delta^{-n} .
\end{eqnarray*}
This and the second inequality in \eqref{e:hat-rho-n-lower} imply that
$$\Omega (\hat \1_n )\lesssim d^{kN} \sum_{m=N+1}^{n-N} e^{m \Omega(\phi)} \delta^{-m}  \hat\rho^-_{n} +
d^{kN} e^{(n-N)\max \phi} d^{kn}\delta^{-n} .$$
Then, by the first inequality in
 \eqref{e:hat-rho-n-lower} and using that $\delta>e^{\Omega(\phi)}$ and $n>2N$, we obtain
$${\Omega (\hat \1_n )\over \hat\rho_n^-} \lesssim d^{kN} \sum_{m=N+1}^{n-N} e^{m \Omega(\phi)}\delta^{-m}  +
d^{kN} e^{(n-N)\Omega(\phi)} \delta^{-n} \lesssim d^{kN} e^{N \Omega(\phi)} \delta^{-N}.$$

Recall that all the constants involved in our computations do not depend on $n$ and $N$.
Since $N$ is chosen large enough, we obtain that 
$\Omega (\hat \1_n )/\hat\rho_n^-\leq d^{kN}/2$. 
This ends the proof of the lemma.
\end{proof}

\begin{lemma} \label{l:eq-1-n}
Under the hypotheses of Theorem \ref{t:main-lam}, for all $J\geq 0$, $N\geq 0$,
and $p>0$, the sequence 
$\|\hat\1_n^*\|_{\log^p}$ is bounded. In particular, the sequence of functions $\hat\1_n^*$
 is equicontinuous.
\end{lemma}

\begin{proof}
We only need to consider $n>2N$, and the implicit
constants below may depend on $N$. 
We will use Lemma \ref{l:ddc-dyn-log} and need to estimate $dd^c \hat \1_n^*$. 
By Lemma \ref{l:theta-bound}
 the sequence $(\hat \theta_n)$ is bounded.
This and Proposition \ref{p:laplace-1-n} imply that
$$\abs{dd^c \hat \1_n^*} \lesssim {1\over \hat\rho^-_n}
\Big( \sum_{m=N+1}^{n-N} \eta(m) e^{m\max \phi}  \hat \rho^-_{n-m} d^{(k-1)m} \omega_{m}
 + \sum_{m=n-N+1}^{n}\eta(m) e^{(n-N)\max \phi} d^{(k-1)m} \omega_{m}\Big).
$$
Then, using
the two inequalities in \eqref{e:hat-rho-n-lower}, we obtain
\begin{eqnarray*}
\abs{dd^c \hat \1_n^*} &\lesssim&  \sum_{m=N+1}^{n-N} \eta(m) e^{m\Omega(\phi)}  d^{-m} \omega_{m}
 + \sum_{m=n-N+1}^{n} \eta(m) e^{(n-N)\Omega(\phi)}  d^{(k-1)m-kn} \omega_{m}\\
&\lesssim& \sum_{m=0}^\infty \eta(m) e^{m\Omega(\phi)}  d^{-m} \omega_{m}.
\end{eqnarray*}
Finally, since $\eta$ is sub-exponential and $e^{\Omega (\phi)}<d$,
Lemma \ref{l:ddc-dyn-log} implies the result.
\end{proof}

\noindent
\proof[End of the proof of Proposition \ref{p:theta-n-1-n}]
By Lemma \ref{l:theta-bound}, we already know that the sequence $(\theta_n)$ is bounded. 
Since $\min\1^*_n=1$, we have $\max\1^*_n=\theta_n$,
hence
the sequence $(\1^*_n)$ is uniformly bounded.
 In order to show that this sequence is equicontinuous, it is enough
 to approximate it uniformly by an equicontinuous sequence.

Take $N=0$. Fix an arbitrary constant $0<\epsilon<1$. 
Since $\|\phi-\phi_m\|_\infty\lesssim m^{-2}$ by \eqref{e:phi-n}, 
we can choose an integer $J$ large enough so that for every $n\geq 0$ we have 
$$(1-\epsilon) \hat\1_n \leq \1_n\leq (1+\epsilon) \hat \1_n.$$
This implies
$${1-\epsilon\over 1+\epsilon} \, \hat\1_n^*\leq \1_n^* \leq {1+\epsilon\over 1-\epsilon} \, \hat\1_n^*.$$
Therefore, $|\1_n^*-\hat\1^*_n|$ is bounded uniformly by a constant times
$\epsilon$. By Lemma \ref{l:eq-1-n}, the sequence $(\hat\1^*_n)$ is
equicontinuous. We easily deduce that the sequence $(\1^*_n)$ is equicontinuous as well.
\endproof

\subsection{Proof of Theorem \ref{t:main-lam}} \label{ss:proof-main-lam} 
We first define the scaling ratio $\lambda$. By definition of $\rho^+_n$, we
easily see that the sequence $(\rho^+_n)$
is sub-multiplicative, that is, 
$\rho^+_{n+m}\leq \rho^+_m \rho^+_n$ for all $m,n\geq 0$. It follows that
the first limit in the following line exists
\[
\lam := \lim_{n\to \infty} \pa{\rho^+_n}^{1/n} = \lim_{n\to \infty}\pa{\rho^-_n}^{1/n} ,
\]
where the last identity is due to the fact that $(\theta_n)$ is bounded, see Lemma \ref{l:theta-bound}. 
We have the following lemma.

\begin{lemma} \label{l:bound-omega-lam}
The sequences $(\lam^{-n} \rho^+_n)$ and $(\lam^{-n}\rho^-_n)$  are
both 
  bounded above and below by positive constants. In particular, the sequence
$\big(\lam^{-n} \1_{n}\big)$ is uniformly bounded and equicontinuous.
\end{lemma}

\begin{proof}
It is clear that the second assertion is a consequence of the first one and
Proposition \ref{p:theta-n-1-n}.  We prove now the first assertion.
Since  the sequence $\rho^+_n$ is sub-multiplicative, it is well-known
that $\inf_n (\rho^+_n)^{1/n}$ is equal to $\lambda$. Hence, we have
$\lambda^{-n}\rho^+_n\geq 1$. Since $\theta_n$ is bounded, we have $\rho_n^+\lesssim \rho_n^-$.
It follows that both $\lambda^{-n}\rho^\pm_n$ are bounded from below by positive constants.
Similarly, the sequence  $\rho^-_n$ is super-multiplicative, i.e., $\rho^-_{n+m}\geq \rho^-_m \rho^-_n$
for all $m,n\geq 0$, and we deduce that  that both $\lambda^{-n}\rho^\pm_n$ are bounded from above
 by positive constants. The lemma follows.
\end{proof}

We can extend the above result to all continuous test functions.

 \begin{lemma} \label{l:Ln-g-equi}
 Let $\Fc$ be a uniformly bounded and equicontinuous family of real-valued functions on $\P^k$. Then the family
 $$\Fc_\N:=\{\lam^{-n}\Ll^n (g) \ : \ g\in\Fc, n\geq 0\}$$
 is also uniformly bounded and equicontinuous.
 \end{lemma}

\begin{proof}
By Lemma \ref{l:bound-omega-lam}, the family
$\Fc_\N$ is uniformly bounded. 
We prove now that it is\break equicontinuous. Given any constant $\epsilon>0$, using
a convolution, we can find for every $g\in\Fc$
 a smooth function $g'$ such that  
$\|g-g'\|_\infty\leq \epsilon$ and $\|g'\|_{\Cc^2}$
is
 bounded by a constant depending on $\epsilon$.
 Denote by $\Fc'$ the family of these $g'$.
Observe that
$$|\lambda^{-n}\Ll^n(g)-\lambda^{-n}\Ll^n(g')|=|\lambda^{-n}\Ll^n(g-g')|
\leq \epsilon \lambda^{-n}\1_n\leq \epsilon \lam^{-n} \rho^+_n$$
and the last expression is bounded by a constant times $\epsilon$. Therefore,
in order to prove the lemma, it is enough to show that  the 
family $\Fc'_\N$, defined in a similar way as for $\Fc_\N$, is equicontinuous.
For simplicity, we replace $\Fc$ by $\Fc'$ and assume that $\|g\|_{\Cc^2}$ 
is bounded by a constant for $g\in\Fc$.
The constants involved in the computation below do not depend on $g\in\Fc$.

We continue to use the notation introduced above. Consider an arbitrary constant $\epsilon>0$. 
Take $N=0$ and choose $J$ large enough depending on $\eps$.
From the definitions of $\Ll$ and $\hat\Ll_n$ (see \eqref{e:l-hat})
and the fact that $\|\phi-\phi_m\|_\infty\lesssim m^{-2}$
we obtain that 
$$|\lambda^{-n}\Ll^n(g)(x)-\lambda^{-n}\hat\Ll_n(g)(x)|\leq \epsilon \lambda^{-n}\sum_{f^n (x) =y } 
e^{ \phi (x) + \phi (f(x) ) + \cdots + \phi (f^{n-1} (x)) }  |g(x)|.$$
This and Lemma \ref{l:bound-omega-lam} imply that
$$\|\lambda^{-n}\Ll^n(g)-\lambda^{-n}\hat\Ll_n(g)\|_\infty
\leq \epsilon \lambda^{-n}\rho_n^+\|g\|_\infty\lesssim \epsilon.$$
So, in order to prove that the family $\lambda^{-n}\Ll^n(g)$ is
equicontinuous, it is enough to show the same property for the family 
$\lambda^{-n}\hat\Ll_n(g)$. 

We will use the same idea as in Proposition \ref{p:laplace-1-n} and Lemma \ref{l:ddch}.
Instead of the function $\h$, we need to consider the following slightly different function (recall that $N=0$)
\[
\H (x_0, \dots, x_n):= e^{\phi_{n+J} (x_0)+\phi_{n+J-1}(x_1) + \dots + \phi_{J+1} (x_{n-1})} g(x_0)=\h(x_0,\ldots,x_n)g(x_0).
\]
We have 
$$i\ddbar \H= (i\ddbar \h)g(x_0) +\h (i\ddbar g(x_0))+i\partial \h\wedge \dbar g(x_0) -i\dbar \h\wedge \partial g(x_0).$$
Applying Cauchy-Schwarz's inequality to the last two terms, and since $g$ has a bounded $\Cc^2$ norm, we obtain
\begin{eqnarray*}
|i\ddbar \H| &\leq&  |(i\ddbar \h)g(x_0)| + |\h (i\ddbar g(x_0))| 
+ i\h^{-1}\partial \h\wedge \dbar \h +i\h \partial g(x_0) \wedge \dbar g(x_0) \\
&\lesssim & |i\ddbar \h| + \h\omega_\FS(x_0) + i\h^{-1}\partial \h\wedge \dbar \h +\h \omega_\FS(x_0) \\
&\lesssim & |i\ddbar \h| + \h\omega_\FS(x_0) + i\h^{-1}\partial \h\wedge \dbar \h.
\end{eqnarray*}

We claim that the last sum satisfies
$$ |i\ddbar \h| + \h\omega_\FS(x_0) + i\h^{-1}\partial \h\wedge \dbar \h \lesssim \h \sum_{m=0}^{n-1} \eta(n-m)\omega^{(m)}.$$
Lemma \ref{l:ddch} shows that the first term $ |i\ddbar \h|$ of the LHS
is bounded by the RHS. The second term clearly satisfies the same property (consider $m=0$ in the above sum).
For the last term, by Cauchy-Schwarz's inequality and using
a computation as in the proof of Lemma \ref{l:ddch}, we have (recall that  $N=0$)
$$i\h^{-1}\partial \h\wedge \dbar \h = 
\h \sum_{m,m'=0}^{n-1} i\partial \phi_{n+J-m}(x_m)\wedge \dbar \phi_{n+J-m'}(x_{m'})
\lesssim  \h \sum_{m=0}^{n-1} \eta(n-m)\omega^{(m)}.$$
This implies the claim and gives a
bound for $|i\ddbar \H|$.

Since $\hat\Ll_n(g)=(\pi_n)_*(\H)$, we obtain as in the proof of Proposition \ref{p:laplace-1-n} that
$$|\ddc \lambda^{-n} \hat\Ll_n(g)|
 \lesssim   \lambda^{-n} \sum_{m=1}^{n} \eta(m) e^{m\max \phi} \hat \rho^+_{n-m} d^{(k-1)m} \omega_{m}.$$
By Lemmas \ref{l:hat-theta-n} and \ref{l:bound-omega-lam}
we have
$\hat\rho_{n-m}^\pm\lesssim \rho_{n-m}^\pm\lesssim \lambda^{n-m}$. Therefore, we obtain
$$|\ddc \lambda^{-n} \hat\Ll_n(g)| \lesssim   \sum_{m=1}^{n} \eta(m) e^{m\max \phi} \lam^{-m} d^{(k-1)m} \omega_{m}.$$
Finally, since $\lam\geq d^k e^{\min\phi}$ by definition of $\lambda$,
 the last estimate implies that
$$|\ddc \lambda^{-n} \hat\Ll_n(g)| \lesssim   \sum_{m=1}^{n} \eta(m) e^{m\Omega(\phi)}  d^{-m} \omega_{m}.$$
Lemma \ref{l:ddc-dyn-log} and the fact that $d>e^{\Omega(\phi)}$ imply the result.
\end{proof}

We now construct the density function $\rho$ on $\P^k$. Recall that the sequence $\lambda^{-n}\1_n$ is uniformly bounded 
and equicontinuous. Therefore, the Cesaro sums
\[
\widetilde\1_n:=\frac{1}{n} \sum_{j=0}^{n-1}
\lam^{-j} \1_j
\]
also form
a uniformly bounded and equicontinuous sequence of functions. 
It follows that there is a subsequence of $\widetilde\1_n$ which converges uniformly 
to a continuous function $\rho$. Observe that $\rho\geq \inf_n \lam^{-n}\rho_n^-$.
Hence, by Lemma \ref{l:bound-omega-lam}, the function $\rho$ is strictly positive. A direct computation gives
$$\lambda^{-1}\Ll(\widetilde\1_n)-\widetilde\1_n= {1\over n} (\lambda^{-n}\1_n - \1_0).$$
Since $\lam^{-n}\1_n$ is bounded uniformly in $n$,
 the last expression tends uniformly to 0 when $n$
tends to infinity. We then deduce from the definition of $\rho$ that $\lam^{-1}\Ll(\rho)=\rho$.

\proof[End of the proof of Theorem \ref{t:main-lam}]
Observe that we only need to show that $\lam^{-n}\Ll^n(g)$ converges to $c_g \rho$ 
for some constant $c_g$. The remaining part of the theorem is then clear.
Let $\Gc$ denote the family of all limit
 functions of subsequences of $\lam^{-n}\Ll^n(g)$. By Lemma \ref{l:Ln-g-equi}, the 
 sequence $\lam^{-n}\Ll^n(g)$ is uniformly bounded and equicontinuous.
 Therefore, by Arzel\`a-Ascoli theorem, $\Gc$ is a uniformly bounded and equicontinuous family of functions which is compact for the uniform topology.
Observe also that $\Gc$ is invariant under the action of $\lambda^{-1}\Ll$.
Define
$$M:=\max\{ l(a)/\rho(a) \ \colon \ l\in\Gc, a\in\P^k \}.$$
We first prove the following properties.

\medskip\noindent
{\bf Claim 1.} We have $\max_{\P^k} (l/\rho)=M$ for every $l\in \Gc$.

\medskip

Assume by contradiction that there is a sequence $\lam^{-n_j}\Ll^{n_j}(g)$ which
converges uniformly to a function $l\in\Gc$ such that $l\leq (M-2\epsilon)\rho$ for
some constant $\epsilon>0$. Then, for $j$ large enough, we have 
$\lam^{-n_j}\Ll^{n_j}(g)\leq (M-\epsilon)\rho$. Fix such an index $j$. For $n>n_j$ we have 
$$\lam^{-n}\Ll^n(g)=\lambda^{-n+n_j} \Ll^{n-n_j} (\lam^{-n_j}\Ll^{n_j}(g)) 
\leq (M-\epsilon)\lambda^{-n+n_j} \Ll^{n-n_j}(\rho)=(M-\epsilon)\rho.$$
Since this is true for every $n>n_j$, we get a contradiction with the definition
of $M$. This ends the proof of Claim 1.

\medskip\noindent
{\bf Claim 2.} We have $l/\rho=M$ on the small Julia set $\supp(\mu)$ for every $l\in\Gc$.

\medskip

Consider an arbitrary function $l\in\Gc$ and define
 $l_n:=\lambda^{-n}\Ll^n(l) \in \Gc$. By Claim 1, there is a point $a_n\in\P^k$ such
that $l_n(a_n)=M\rho(a_n)$. 
By definition of $M$, we have $l\leq M\rho$ and hence 
$$M\rho(a_n)=l_n(a_n)=\lambda^{-n}\Ll^n(l)(a_n)\leq \lambda^{-n}\Ll^n(M\rho)(a_n)=M\rho(a_n).$$ 
So the inequality in the last line 
is actually an equality. This and the definition of $\Ll$ imply that $l/\rho=M$ on $f^{-n}(a_n)$. 
Observe that when $n$ tends to infinity, the limit of $f^{-n}(a_n)$ contains
$\supp(\mu)$,
 see, e.g., \cite[Cor.\,1.4]{dinh2010equidistribution}.
By continuity, we obtain $l/\rho=M$ on
$\supp(\mu)$.
 This ends the proof of Claim 2. 

\smallskip

Applying the above claims to the function $-g$ instead of $g$, we obtain that $l/\rho$ is equal
on $\supp(\mu)$ to
$\min_{\P^k} (l/\rho)$. 
 We can
now conclude that $l=M\rho$ on $\P^k$
 for every $l\in\Gc$. Define $c_g:=M$. We obtain that $\lambda^{-n}\Ll^n(g)$ converges uniformly to $c_g\rho$. 
This completes the proof of the theorem.
\endproof

\section{Properties of equilibrium states}
\label{s:new-properties}

 In this section we conclude the proof of Theorem
\ref{t:main}.
In Sections
\ref{s:equidistributions-L}  and \ref{s:proof-t:main-further},
we deduce 
the main properties of the equilibrium states in
Theorem \ref{t:main} from
Theorem \ref{t:main-lam}.
 In Section \ref{s:pern} we prove the equidistribution of 
 repelling
 periodic points, which concludes the proof of Theorem \ref{t:main}.

\subsection{Equidistribution of preimages and mixing properties}\label{s:equidistributions-L}
We have seen that the operator $\Ll$ acts on the space of continuous
functions $g \colon \P^k\to \R$. 
It is also positive, i.e., $\Ll(g)\geq 0$ when $g\geq 0$. 
Therefore, $\Ll$ induces by duality a linear operator $\Ll^*$ acting on
the space of measures and preserving the cone of positive measures. 

\begin{proposition} \label{p:cv-m-phi}
Under the assumptions of Theorem \ref{t:main}, there exists a unique
conformal measure associated with $\phi$, that is, there exists a unique
probability measure $m_\phi$ which is an eigenvector of $\Ll^*$. We also
have $\Ll^*(m_\phi)=\lambda m_\phi$, $\supp(m_\phi)=\supp(\mu)$, 
and if $\nu$ is a positive measure, $\lambda^{-n}(\Ll^{n})^{*}(\nu)$
converges to $\langle\nu,\rho\rangle m_\phi$ when $n$ tends to infinity.
Moreover, if $\Fc$ is a uniformly bounded and equicontinuous family
of functions on $\P^k$, then 
$\lambda^{-n}\Ll^n(g)-c_g\rho$ converges to $0$ when $n$ goes
to infinity, uniformly on $g\in\Fc$,
 where $c_g := \sca{m_\phi, g}$.
\end{proposition}
\proof
For any probability
measure $m_\phi$ as in the first assertion, 
there is a constant $\lambda'>0$ such that $\Ll^* (m_\phi)=\lambda' m_\phi$. It follows that,
for every continuous function $g$,
$$\langle m_\phi,g\rangle = \lim_{n\to\infty} \langle \lambda'^{-n}(\Ll^{n})^* (m_\phi),g\rangle 
=  \lim_{n\to\infty} \langle m_\phi,\lambda'^{-n}\Ll^n(g)\rangle.$$
We necessarily have $\lambda'=\lambda$ because we know 
from the end of the proof of Theorem \ref{t:main-lam} 
that $\lambda^{-n}\Ll^n(g)$ converges uniformly to $c_g\rho$ and $c_g$ is not always 0. 
We conclude that $\langle m_\phi,g\rangle =c_g\langle m_\phi,\rho\rangle$.
Since $\langle m_\phi,g\rangle =c_g=1$ when $g=\1$ (because $m_\phi$
is a probability measure) we deduce that $\langle m_\phi,\rho\rangle=1$
and hence $\langle m_\phi,g\rangle=c_g$ for every continuous function $g$. 
This gives the uniqueness of $m_\phi$.

Consider now an arbitrary probability measure $\nu$ on $\P^k$. We have
$$\langle \lambda^{-n}(\Ll^{n})^*(\nu),g\rangle 
=\langle \nu, \lam^{-n} \Ll^n(g)\rangle \rightarrow \langle\nu,c_g\rho\rangle 
=\langle\nu,\rho\rangle \langle m_\phi,g\rangle.$$
It follows that $\lambda^{-n}(\Ll^{n})^*(\nu)$ converges to $\langle\nu,\rho\rangle m_\phi$.
If $\nu$ is
supported by $\supp(\mu)$ and $g$ vanishes on $\supp(\mu)$, by 
definition of $\Ll$, the function $\Ll^n(g)$ also vanishes on $\supp(\mu)$ 
and the last computation implies that $\langle m_\phi,g\rangle=0$. 
Equivalently, the measure $m_\phi$ is supported by $\supp(\mu)$. 

In order to show that $\supp(m_\phi)=\supp(\mu)$, we assume by 
contradiction that there is a continuous function $g\geq 0$ on $\P^k$ such 
that $g>0$ on some open subset $U$ of $\supp(\mu)$ and $\langle m_\phi,g\rangle=0$.
The
$\lam^{-1}\Ll^*$-invariance
of $m_\phi$ implies that 
$\langle m_\phi, \Ll^ng\rangle = \lambda^n\langle m_\phi,g\rangle =0$ 
and the definition of $\Ll$ implies that $\Ll^n(g)>0$ on $f^n(U)$.
It follows that $m_\phi$ has no mass on $f^n(U)$ and hence on $\cup_{n\geq 0} f^n(U)$. On the 
other hand, we have for every $x\in \P^k$ that $d^{-kn} (f^n)^*(\delta_x)$ converges to $\mu$, see, e.g.,
\cite[Cor.\,1.4]{dinh2010equidistribution}.
Therefore, $f^{-n}(\delta_x)\cap U\neq\varnothing$ for some $n$ or equivalently $x\in \cup_{n\geq 0} f^n(U)$.
So we have $\cup_{n \geq 0} f^n (U) = \P^k$.
This contradicts the fact that $m_\phi$
has no mass on this union. So we have $\supp(m_\phi)=\supp(\mu)$ as desired.

For the last assertion of
 the proposition, we can replace $g$ with $g-c_g\rho$ in order to assume that $c_g=0$ for $g\in\Fc$. 
By Lemma \ref{l:Ln-g-equi}, the family $\Fc_\N$ is uniformly 
bounded and equicontinuous. So the limit of the sequence of 
sets $\lam^{-n}\Ll^n(\Fc)$ is a compact, uniformly bounded and 
equicontinuous family of functions that we denote by $\Fc_\infty$. 
This family is invariant by $\lam^{-1}\Ll$ and we also have $c_g=0$ 
for $g\in\Fc_\infty$. We want to show that it contains only the function 0.

Define
$$M:=\max\{l(a)/\rho(a) \ \colon \  l\in\Fc_\infty, a\in\P^k\}.$$
Choose a function $l\in \Fc_\infty$ and a point $a$ such that 
$l(a)/\rho(a)=M$. There are an increasing sequence of integers 
$(n_j)$ and a sequence $(g_j)\subset\Fc$ such that $\lam^{-n_j}\Ll^{n_j}(g_j)$ 
converges uniformly to $l$. For every $n\geq 0$, choose a limit 
function $l_{-n}$ of the sequence $\lam^{-n_j+n}\Ll^{n_j-n}(g_j)$. 
We have $l=\lam^{-n} \Ll^n(l_{-n})$ and $l_{-n}\in\Fc_\infty$. 

As in the end of the proof of Theorem \ref{t:main-lam},
 we obtain that $l_{-n}/\rho=M$ on the set $f^{-n}(a)$ and if $l_{-\infty}$ 
 is a limit of the sequence $l_{-n}$ then $l_{-\infty}$ belongs to $\Fc_\infty$ 
 and $l_{-\infty}/\rho=M$ on the small Julia set $\supp(\mu)$. Since $m_\phi$ 
 is supported by the small Julia set and $\langle m_\phi,g\rangle =c_g=0$ 
 for $g\in\Fc_\infty$, we conclude that $M=0$. Using the same argument
  for $-g$ with $g\in\Fc$, we obtain that the minimal value of the functions 
  in $\Fc_\infty$ is also 0. So $\Fc_\infty$ contains only the function 0.
This ends the proof of the proposition.
\endproof

Proposition \ref{p:cv-m-phi} in particular gives the following
equidistribution result
for the (weighted) preimages of a given point.

\begin{corollary}\label{c:equi-preim}
Under the assumptions
 of Theorem \ref{t:main}, 
for every $x\in\P^k$
 the points in $f^{-n}(x)$, with suitable weights, 
are equidistributed with respect to the conformal measure $m_\phi$ 
when $n$ tends to infinity. More precisely, if $\delta_a$ denotes the Dirac mass at $a$, then
$$\lim_{n\to\infty} \lambda^{-n} \sum_{f^n(a)=x} e^{\phi(a)+\cdots+\phi(f^{n-1}(a))} \delta_a = \rho(x) m_\phi$$
for every $x\in \P^k$.
\end{corollary}
\proof
Denote by $\mu_n$ the measure in the LHS of the last identity.
Let $g$ be any 
continuous function on $\P^k$. We have
$$\langle \mu_n,g\rangle
= \lambda^{-n}  \sum_{f^n(a)=x} e^{\phi(a)+\cdots+\phi(f^{n-1}(a))} g(a) = \lambda^{-n} (\Ll^ng)(x).$$
The last expression converges to $c_g\rho(x)=\rho(x)\langle m_\phi,g\rangle$. The result follows.
\endproof

For our convenience, define the operator $L$ by 
$L(g):=(\lam \rho)^{-1}\Ll(\rho g)$. Define also the positive measure $\mu_\phi$ 
by $\mu_\phi:=\rho m_\phi$. We have the following lemma.

\begin{lemma} \label{l:def-mu-phi}
For any continuous function $g\colon \P^k\to \R$, 
the sequence 
$L^n(g)$ converges
uniformly to the constant $c_{\rho g} = \langle \mu_\phi,g\rangle=\langle m_\phi,\rho g\rangle$. 
We also have that $\mu_\phi$ is an $f$-invariant probability measure
such that $\supp(\mu_\phi)=\supp(\mu)$. 
\end{lemma}
\proof
Define $g':=\rho g$. We have $c_{g'}= \langle m_\phi, \rho g \rangle = \langle\mu_\phi,g\rangle$.
The first assertion is 
a direct
consequence of the fact that $\lam^{-n}\Ll^n(g')$ converges uniformly to $c_{g'}\rho$.

For the second assertion, 
we have seen in the proof of Proposition \ref{p:cv-m-phi} that $\langle m_\phi,\rho\rangle=1$. 
It follows that $\mu_\phi$ is a probability measure. Moreover, we 
obtain from the $\lam^{-1}\Ll^*$-invariance
of $m_\phi$ that
$$\langle \mu_\phi,g\circ f\rangle=\langle m_\phi,\rho(g\circ f)\rangle=\langle \lambda^{-1}\Ll^*(m_\phi),\rho(g\circ f)\rangle
=\langle m_\phi, \lam^{-1} \Ll(\rho(g\circ f))\rangle=\langle \mu_\phi, L(g\circ f)\rangle.$$
Using that $\lam^{-1}\Ll(\rho)=\rho$ and the definition of $\Ll$, we can easily check that 
$L(g\circ f)=g$. So the previous identities 
imply that $\langle \mu_\phi,g\circ f\rangle=\langle\mu_\phi,g\rangle$. 
Hence, $\mu_\phi$ is an invariant measure. The assertion on the 
support of $\mu_\phi$ is clear because $\supp(m_\phi)=\supp(\mu)$ by Proposition \ref{p:cv-m-phi} and $\rho$
is strictly positive.  
\endproof

The operator 
$\Ll$ can also be extended to a continuous operator on
$L^2 (\mu_\phi)$ and $L^2 (m_\phi)$. Since $\mu_\phi=\rho m_\phi$
and $\rho$ is positive and continuous, these two spaces are actually the same and the corresponding norms are equivalent. 

\begin{lemma}\label{l:norm-L2-mu-phi}
Under the assumptions of Theorem \ref{t:main}, the operator
$\Ll$ extends to a linear continuous operator on 
$L^2 (m_\phi)$ whose norm is
bounded by $\lambda e^{{1\over 2}\Omega(\phi)}$.
Moreover, there exists a positive constant 
$c$ such that $\norm{\lam^{-n}\Ll^n}_{L^2 (m_\phi)}\leq c$ for all $n\geq 0$.
\end{lemma}

\begin{proof}
By Cauchy-Schwarz's 
inequality and using the $\lam^{-1}\Ll^*$-invariance of $m_\phi$, we have
\[
\big\langle m_\phi,  \abs{\Ll^n g}^2 \big\rangle \leq
\big\langle  m_\phi , ({\Ll^n \1}) \cdot (\Ll^n \abs{g}^2) \big\rangle \leq \rho_n^+ \big\langle  m_\phi , \Ll^n \abs{g}^2 \big\rangle
= \rho_n^+\lambda^{n} \big\langle  m_\phi , \abs{g}^2 \big\rangle
\] 
for every $g\in 
L^2 (m_\phi)$ and $n\geq 0$. 
The second assertion of the lemma follows because $\rho_n^+\lesssim \lambda^n$.

For the first assertion, take $n=1$. From the definition of $\rho_n^+$ and $\lambda$, we have $\rho_1^+\leq d^ke^{\max\phi}$ and $\lambda\geq d^ke^{\min\phi}$.
The above inequality implies that 
$$\big\langle m_\phi,  \abs{\Ll g}^2 \big\rangle \leq e^{\Omega(\phi)}\lambda^2 \big\langle  m_\phi , \abs{g}^2 \big\rangle.$$
The first assertion in the lemma follows.
\end{proof}

\begin{proposition}\label{p:mixing}
Under the assumptions of Theorem \ref{t:main},
the measure $\mu_\phi = \rho m_\phi$ is K-mixing and mixing of all orders.
\end{proposition}

\proof
We start with the second property.
Let $\{g_0,\ldots,g_r\}$ be any finite family of continuous 
test functions on $\P^k$. We need to show 
that,
for $0=n_0\leq n_1\leq\cdots \leq n_r$,
$$\big\langle \mu_\phi,g_0(g_1\circ f^{n_1})\ldots (g_r\circ f^{n_r})\big\rangle
 - \prod_{j=0}^r \langle \mu_\phi, g_j\rangle \to 0$$
when $n:=\inf_{0\leq j<r} (n_{j+1}-n_j)$ tends to infinity.
This property is clearly true for $r=0$. Take $r\geq 1$. By induction, 
we can assume that the above convergence holds for the case of $r-1$ 
test functions. We prove now the same property for $r$ test functions.

By the $f_*$-invariance of $\mu_\phi$ and the induction hypothesis, we have 
$$\big\langle \mu_\phi,(g_1\circ f^{n_1})\ldots (g_r\circ f^{n_r})\big\rangle = 
\big\langle \mu_\phi, g_1(g_2\circ f^{n_2-n_1})\ldots (g_r\circ f^{n_r-n_1})\big\rangle
 \rightarrow  \prod_{j=1}^r \langle \mu_\phi, g_j\rangle.$$ 
So the desired property holds when $g_0$ is a constant function. Therefore, 
we can subtract from $g_0$ a constant and assume
 that $\langle \mu_\phi, g_0\rangle=0$, which implies that the 
 product $\prod_{j=0}^n \langle\mu_\phi,g_j\rangle$ vanishes.  
Using that $\lam^{-1}\Ll(\rho)=\rho$, the $\lam^{-1}\Ll^*$-invariance
of $m_\phi$, and the definition of $\Ll$, we can easily check by induction that for all functions $g,l$ we have
$$\langle \mu_\phi,g\rangle =
  \langle \mu_\phi, L^n(g)\rangle \qquad \text{and} \qquad L^n(g(l\circ f^n))=L^n(g)l.$$
We then deduce that
\begin{eqnarray*}
\big\langle \mu_\phi,g_0(g_1\circ f^{n_1})\ldots (g_r\circ f^{n_r})\big\rangle 
&=& \big\langle \mu_\phi, L^{n_1}\big(g_0(g_1\circ f^{n_1})\ldots (g_r\circ f^{n_r})\big)\big\rangle \\
&=& \big\langle \mu_\phi, L^{n_1}(g_0) g_1\ldots (g_r\circ f^{n_r-n_1})\big\rangle.
\end{eqnarray*}
By Lemma \ref{l:def-mu-phi}, the sequence
$L^{n_1}(g_0)$
converges uniformly to 0
as $n_1$ tends to $\infty$. 
So the last integral tends to 0 because the function
 $g_1\ldots (g_r\circ f^{n_r-n_1})$ is bounded. We then 
 conclude that $\mu_\phi$ is mixing of all orders.

We prove now that $\mu_\phi$ is K-mixing, that is,
 that given $g \in L^2 (\mu_\phi)$,
when $n$ tends to infinity the difference
$$\big\langle \mu_\phi, g(l\circ f^n)\big\rangle - \langle \mu_\phi,g\rangle \langle \mu_\phi, l\rangle$$
tends to 0 uniformly on test functions 
$l$ whose $L^2(\mu_\phi)$-norm is bounded by a constant.
As above, we can assume that $\langle\mu_\phi,g\rangle=0$. We 
can also assume that the $L^2(\mu_\phi)$-norms of $g$ and $l$ 
are bounded by 1. 
Fix an arbitrary constant $\epsilon>0$. It is enough to show the
 existence of an integer $N=N(\epsilon)$ independent of $l$ such that 
$|\langle \mu_\phi, g(l\circ f^n)\rangle|\leq 2\epsilon$ for $n\geq N$. 

Choose a continuous function $g'$ such that $\langle \mu_\phi,g'\rangle=0$ and $\|g-g'\|_{L^2(\mu_\phi)}\leq \epsilon$.
Using the invariance of $\mu_\phi$ we have
\begin{eqnarray*}
|\langle \mu_\phi, g(l\circ f^n)\rangle-\langle \mu_\phi, g'(l\circ f^n)\rangle| 
&=& |\langle \mu_\phi, (g-g')(l\circ f^n)\rangle| \leq \|g-g'\|_{L^2(\mu_\phi)}
\|l\circ f^n\|_{L^2(\mu_\phi)} \\
&=& \|g-g'\|_{L^2(\mu_\phi)} \|l\|_{L^2(\mu_\phi)} \leq \epsilon.
\end{eqnarray*}
It remains to show that $|\langle \mu_\phi, g'(l\circ f^n)\rangle|\leq \epsilon$ 
when $n\geq N$ for some $N$ large enough. 
As above, we have 
$$|\langle \mu_\phi, g'(l\circ f^n)\rangle|
 = |\langle \mu_\phi, L^n(g')l\rangle| \leq \| L^n(g')\|_\infty.$$
Lemma \ref{l:def-mu-phi} and the identity $\langle \mu_\phi,g'\rangle=0$ imply the result.
\endproof

For positive real numbers $q,M$, and $\Omega$ with $q>2$ and $\Omega<\log d$, consider the following set of weights
$$\Pc(q,M,\Omega):=\big\{\phi\colon\P^k\to\R \ :   \	\|\phi\|_{\log^q}\leq M,\ \Omega(\phi)\leq \Omega\big\}$$  
and the uniform topology induced by the sup norm. Observe that this family is equicontinuous. 
In the two lemmas below, we study the dependence on $\phi\in\Pc(q,M,\Omega)$
of the objects introduced in this section. 
Therefore, we will use the index $\phi$ or parameter $\phi$ for objects which depend on $\phi$, e.g., we will 
write $\lambda_\phi, \Ll_\phi, \rho_\phi, \1_n(\phi)$ instead of $\lambda, \Ll, \rho$ and $\1_n$.

\begin{lemma} \label{l:dependence}
Let $q,M$, and $\Omega$ be positive real numbers 
such that $q>2$ and $\Omega<\log d$.
The maps $\phi\mapsto\lambda_\phi$, $\phi\mapsto m_\phi$, $\phi\mapsto \mu_\phi$,
and $\phi\mapsto \rho_\phi$ are continuous on $\phi\in \Pc(q,M,\Omega)$ with
respect to the standard topology on $\R$, the
weak topology on measures, and the uniform topology on functions. 
In particular, $\rho_\phi$ is bounded from above and below by positive
constants which are independent of $\phi\in \Pc(q,M,\Omega)$.
Moreover, $\|\lambda_\phi^{-n}\Ll_\phi^n\|_\infty$ is bounded by a
constant which is independent of $n$ and of $\phi\in \Pc(q,M,\Omega)$.
\end{lemma}
\proof
Fix $q,M,$ and $\Omega$ as above. Observe that when we add to $\phi$
a constant $c$ the scaling ratio $\lambda_\phi$ and the operator $\Ll_\phi$
are both changed by a factor $e^c$.
It follows that the
operator $\lambda_\phi^{-1}\Ll_\phi$, the measures $m_\phi$, $\mu_\phi$,
and the density function $\rho_\phi$ do not change. So, for simplicity, it is
enough to prove the lemma for $\phi$ in the family 
$$\Pc_0(q,M,\Omega):=\big\{\phi\colon\P^k\to\R \ :
 \ \min \phi=0,\ \|\phi\|_{\log^q}\leq M,\ \Omega(\phi)\leq \Omega\big\}.$$  
Notice that this family is compact for the uniform topology.

Consider two weights $\phi$ and $\phi'$ in this space. From the definition of $\lambda_\phi$ and $\lambda_{\phi'}$,
we  have $e^{-\|\phi-\phi'\|_\infty} \leq \lambda_\phi/\lambda_{\phi'} \leq e^{\|\phi-\phi'\|_\infty}$.
It follows that $\phi\mapsto\lambda_\phi$ is continuous. When $\phi'\to\phi$,
any limit value of $m_{\phi'}$ is a probability measure invariant by
$\lambda_\phi^{-1}\Ll_\phi^*$ thanks to the invariance of $m_{\phi'}$
by $\lambda_{\phi'}^{-1}\Ll_{\phi'}^*$. Since $m_\phi$ is the only probability
measure which is invariant by $\lambda_\phi^{-1}\Ll_\phi^*$, this limit
value must be $m_\phi$. Thus, $\phi\mapsto m_{\phi'}$ is continuous.

We deduce from the proof of Proposition \ref{p:theta-n-1-n} that
$\theta_n(\phi)=\rho^+_n(\phi)/\rho^-_n(\phi)$ is bounded by a constant independent of $n$ and $\phi$. Moreover, 
the family of functions
$$\big\{\1_n^*(\phi) \quad \text{with} \quad n\geq 0 \quad \text{and} \quad  \phi\in \Pc_0(q,M,\Omega)\big\}$$
is uniformly bounded and equicontinuous. Recall that $\1_n^*(\phi)=(\rho^-_n(\phi))^{-1}\1_n(\phi)$
and $\rho_n^-(\phi)\leq\lambda_\phi^n\leq \rho_n^+(\phi)$, see the proof of Lemma
\ref{l:bound-omega-lam}. It follows that $\lambda_\phi^{-n}\1_n(\phi)$
belongs to a  uniformly bounded and equicontinuous family of functions.

From the definition of $\rho_\phi$ and $\rho_{\phi'}$, we also see that these functions
belong to a uniformly bounded and equicontinuous family of functions. When $\phi'\to\phi$,
if $\rho'$ is any limit of $\rho_{\phi'}$, then $\rho'$ is continuous and invariant by $\lambda_\phi^{-1}\Ll_\phi$ because 
$\rho_{\phi'}$ satisfies a similar property. It follows from Theorem \ref{t:main-lam} that
$\rho'=c\rho_\phi$ for some constant $c$. On the other hand, since
$\mu_{\phi'}=\rho_{\phi'} m_{\phi'}$ is a probability measure, any limit of
$\rho_{\phi'} m_{\phi'}$ is a probability measure. Thus, $\rho' m_\phi=c\mu_\phi$
is a probability measure and hence $c=1$. We conclude that $\rho_{\phi'}\to\rho_\phi$
and also $\mu_{\phi'}\to\mu_\phi$. In other words, the maps $\phi\mapsto \mu_\phi$
and $\phi\mapsto \rho_\phi$ are continuous. Since $\rho_\phi$ is strictly positive and the family
$\Pc_0(q,M,\Omega)$ is compact, we deduce that $\rho_\phi$ is bounded from above
and below by positive constants independent of $\phi$.

The last assertion in the lemma is also clear because
$\|\lambda_\phi^{-n}\Ll_\phi^n\|_\infty=\lambda_\phi^{-n}\|\1_n(\phi)\|_\infty\leq \theta_n(\phi)$.
This ends the proof of the lemma.
\endproof

\begin{lemma} \label{l:dependence-bis}
Let $q,M$, and $\Omega$ be positive real numbers 
such that $q>2$ and $\Omega<\log d$.
Let $\Fc$ be a uniformly bounded and equicontinuous family of
real-valued functions on $\P^k$. Then the family
$$\big\{\lambda_\phi^{-n} \Ll_\phi^n(g) \ : \  n\geq 0, \ \phi\in \Pc(p,M,\Omega), \ g\in\Fc\big\}$$
is equicontinuous. Moreover,
$\big\|\lambda_\phi^{-n} \Ll_\phi^n(g)-\langle m_\phi,g\rangle\big\|_\infty$
tends to $0$ uniformly on $\phi\in \Pc(p,M,\Omega)$ and $g\in\Fc$ when $n$ goes to infinity.
\end{lemma}
\proof
As in Lemma \ref{l:dependence}, we can assume that $\phi\in \Pc_0(p,M,\Omega)$. 
The first assertion is clear from the proof of Lemma \ref{l:Ln-g-equi}.
We prove now the second assertion. 
By Lemma \ref{l:dependence}, $m_\phi$ belongs to a compact
family of probability measures. It follows that $|\langle m_\phi,g\rangle|$ is
bounded by a constant independent of $\phi$ and $g$.
It follows that the family 
$$\Fc_\N':=\big\{\lambda_\phi^{-n} \Ll_\phi^n(g)-\langle m_\phi,g\rangle
\ : \  
n\geq 0,\ \phi\in \Pc_0(p,M,\Omega), \ g\in\Fc\big\}$$ 
is uniformly bounded and equicontinuous. Denote by
$\Fc_\infty'$ the set of all functions $l'$ obtained as the limit of a sequence
$$h_j:=\lambda_{\phi_j}^{-n_j} \Ll_{\phi_j}^{n_j}(g_j)-\langle m_{\phi_j},g_j\rangle$$
in $\Fc_\N'$ with $n_j\to\infty$. By taking a subsequence,
we can assume that $\phi_j$ converges uniformly to some function 
$\phi\in \Pc_0(p,M,\Omega)$. Since $\langle m_{\phi_j},h_j\rangle =0$,
we also obtain that $\langle m_\phi,l'\rangle=0$ by the continuity of $\phi\mapsto m_\phi$.
Now, as in the end of the proof of Proposition
\ref{p:cv-m-phi}, we obtain that $l'=\lam_\phi^{-n} \Ll_\phi^n(l'_{-n})$ for some $l_{-n}'\in\Fc_\infty'$ and then deduce that 
$l'=0$. The lemma follows.
\endproof

\subsection{Pressure and uniqueness of the equilibrium state}
\label{s:proof-t:main-further}
Using the results in the previous section, 
to  prove the 
next proposition 
we only need to follow the arguments in
\cite[Sections 6 and 7]{urbanski2013equilibrium} and \cite[Section 5.6]{przytycki2010conformal}.

\begin{proposition}\label{p:pu-uz}
The probability measure $\mu_\phi$
is a unique 
equilibrium state associate to $\phi$. Moreover, the pressure $P(\phi)$ is equal to $\log \lam$.
\end{proposition}

\begin{proof}
We follow the approach in \cite[Th. 5.6.5]{przytycki2010conformal}. 
To simplify the notation,
set $S_n (g) :=\sum_{j=0}^{n-1} g \circ f^j $
 for any function $g \colon \P^k \to \R$.
 Recall that, given $\phi'\colon \P^k \to \R$ with  $\norm{\phi'}_{\log^q}<\infty$ and $\Omega (\phi')<\log d$, 
 we 
 denote by $\lam_{\phi'},\rho_{\phi'}$
 the objects 
 associated to $\Ll_{\phi'}$.

\medskip\noindent
{\bf Claim 1.} We have $\Ent{f}{\mu_{\phi'}} + \sca{\mu_{\phi'}, \phi'} =P(\phi') = \log \lambda_{\phi'}$ for all $\phi'\colon \P^k \to \R$
such that $\norm{\phi'}_{\log^q}<\infty$ and $\Omega (\phi')<\log d$.
\begin{proof}[Proof of Claim 1]
The proof of the 
inequality $P(\phi')\leq \log \lambda_{\phi'}$
is an adaptation of
 Gromov's proof of the fact
that the topological entropy of $f$ is bounded above by $k \log d$, see \cite{gromov2003entropy}. 
We refer to \cite[Th.\ 6.1]{urbanski2013equilibrium} for the complete details. To complete the proof, it
is enough to show that $\Ent{f}{\mu_{\phi'}} + \sca{\mu_{\phi'}, \phi'} \geq \log \lambda_{\phi'}$.

It follows from \cite{parry1969entropy} 
that $\Ent{f}{\mu_{\phi'}}\geq \big\langle\mu_{\phi'}, \log J_{\mu_{\phi'}}\big\rangle$, 
where
$J_{\mu_{\phi'}}$ is defined as
the Radon-Nikodym derivative of $f^* \mu_{\phi'}$ with respect to $\mu_{\phi'}$
(when this derivative exists).
In our setting, it follows
from a straightforward computation that $J_{\mu_{\phi'}}$ is well defined and given by
\[
J_{\mu_{\phi'}} =
\lambda_{\phi'} 
\rho_{\phi'}^{-1} e^{-\phi'} ( \rho_{\phi'} \circ f).
\]
Indeed, denoting by $J'$ the RHS in the above expression,
for every continuous function $g \colon \P^k \to \R$, we have
\[\begin{aligned}
\big\langle
\mu_{\phi'}, J' g\big\rangle 
& = \big\langle
\lambda m_{\phi'} , e^{-\phi'} (\rho_{\phi'} \circ f) g\big\rangle
=\big\langle\Ll_{\phi'}^* m_{\phi'} ,  e^{-\phi'} (\rho_{\phi'} \circ f)g\big\rangle
=\big\langle m_{\phi'}, \Ll_{\phi'} ( e^{-\phi'} (\rho_{\phi'} \circ f) g)\big\rangle\\
& =\big\langle m_{\phi'}, \rho_{\phi'}  \Ll_{\phi'} (e^{-\phi'} g) \big\rangle
= \big\langle \mu_{\phi'}, f_* g \big\rangle = \big\langle f^* \mu_{\phi'}, g \big\rangle,
\end{aligned}\]
which proves that $J'=J_{\mu_{\phi'}}$.
We then have, using the $f_*$-invariance of $\mu_{\phi'}$,
\[
\Ent{f}{\mu_{\phi'}} + \sca{\mu_{\phi'}, \phi'}
 \geq 
\big\langle
\mu_{\phi'}, \log J_{\mu_{\phi'}}\big\rangle 
+  \sca{\mu_{\phi'}, \phi'}
=
\sca{\mu_{\phi'}, \log (\rho_{\phi'}\circ f) - \log \rho_{\phi'}} 
+ \log \lam_{\phi'}
 = \log \lambda_{\phi'}
\]
and the proof is complete.
\end{proof}

\medskip\noindent
{\bf Claim 2.} 
Let $M$ and $\Omega$
be positive real numbers such that 
 $\Omega < \log d$, and $g\colon \P^k \to \R$ a continuous function. 
Then, for every $y \in \P^k$, we have
\begin{equation}\label{e:c:0}
\frac{1}{n} \frac{\sum_{f^n (x)=y}  
S_n (g) (x)
 e^{S_n (\phi') (x) } }{\Ll^n_{\phi'} \1  (y)}
\to \sca{\mu_{\phi'}, g}
\end{equation}
where the convergence is uniform on 
 $\phi' \in \Pc (q,M,\Omega)$.

\begin{proof}[Proof of Claim 2] 
Observe that the LHS of \eqref{e:c:0} is equal to
\[
\frac{1}{n} \frac{\lam_{\phi'}^{-n} \sum_{f^n (x)=y}  
S_n (g) (x)
 e^{S_n (\phi') (x) } }{\lam^{-n}_{\phi'}\Ll^n_{\phi'} \1  (y)}.
\]
The denominator of the last quotient converges to $\rho_{\phi'}(y)$ and 
the numerator satisfies
\begin{equation}\label{e:c:1}
\lam_{\phi'}^{-n}
\sum_{f^n (x)=y} 
S_n (g) (x)
 e^{
S_n (\phi') (x) 
 } =
 \lam_{\phi'}^{-n}
 \sum_{j=0}^{n-1} \Ll_{\phi'}^n (g \circ f^j) (y)
 =
 \lam_{\phi'}^{-j}
 \sum_{j=0}^{n-1} \lam_{\phi'}^{j-n} \Ll_{\phi'}^{n-j}
  (g \cdot \Ll^j_{\phi'}\1 ) (y).
\end{equation}
 It follows
from Lemma 
\ref{l:dependence-bis} 
that
\begin{equation}\label{e:c:2}
\lam_{\phi'}^{j-n} \Ll_{\phi'}^{n-j} (g \cdot \Ll^j_{\phi'} \1 ) \to 
\big\langle
m_{\phi'}, g \cdot \Ll_{\phi'}^j \1
\big\rangle 
\rho_{\phi'}
\end{equation}
as $n-j \to \infty$, where the convergence is uniform on
$\phi' \in \Pc (q',M,\Omega)$.
We deduce from \eqref{e:c:1}, \eqref{e:c:2},
and the fact that $\lam_{\phi'}^{-j}\Ll_{\phi'}^{j} \1 \to \rho_{\phi'}$ 
as $j\to \infty$ 
 that, as $n\to \infty$,
  the LHS in \eqref{e:c:0} tends to
\[ 
\lim_{n\to \infty}
\frac{1}{n}
\sum_{j=0}^{n-1}
\lam_{\phi'}^{-j} \big\langle
m_{\phi'}, g \cdot  \Ll^j_{\phi'} \1 \big\rangle 
= \sca{m_{\phi'}, g \cdot \rho_{\phi'}} 
= \sca{\mu_{\phi'}, g}.\]
The proof is complete.
\end{proof}

\medskip\noindent
{\bf Claim 3.} For every 
$\psi\colon \P^k \to \R$
 such that $\norm{\psi}_{\log^{q}}<\infty$
the function $t \mapsto P(\phi + t \psi)$
is differentiable
in a neighbourhood of $0$.
\begin{proof}[Proof of Claim 3]  Fix $y \in \P^k$ and 
set
\[
P_n (t) :=
\frac{1}{n}\log \Ll^n_{\phi + t \psi} \1 (y)
\quad  \mbox{ and }\quad  \quad Q_n(t) := \frac{d}{dt} P_n (t)
= \frac{1}{n}
\frac{
\sum_{f^n (x)=y}  
S_n (\psi) (x)
 e^{S_n (\phi+t\psi) (x) }
}{ \Ll^n_{\phi + t \psi} \1 (y)}.
\]
Notice that $\Omega (\phi + t \psi)<\log d$ for $t$ sufficiently small.
A direct computation and Claim 2 
(applied with $\phi+t\psi, \psi$
 instead of $\phi',g$)
imply that $Q_n (t) \to \sca{\mu_{\phi + t \psi},\psi} $ as $n\to \infty$, 
locally uniformly with respect 
to $t$. We also have $P_n (t) \to \log \lambda_{\phi+t\psi}= P(\phi+ t \psi)$, where
the convergence follows from
Lemma \ref{l:bound-omega-lam}
and 
the equality from Claim 1 applied with $\phi'$ instead of $\phi+t\psi$.
We deduce that the pressure function $P$, in a neighbourhood of $t=0$,
is the uniform limit of the $\Cc^1$ functions $P_n (t)$,
whose derivatives $Q_n (t)$
 are also uniformly convergent. Thus, the function
$P$
is differentiable in a neighbourhood of
$t=0$, with derivative at $t$ equal to $\sca{\mu_{\phi+t\psi},  \psi}$.
\end{proof}

It follows from Claim 1 
 that $\mu_\phi$ is an equilibrium state.
By \cite[Cor.\ 3.6.7]{przytycki2010conformal},
the fact that 
the pressure function $t \mapsto P(\phi + t \psi)$ is differentiable
 at $t=0$ with respect to a dense set of continuous functions $\psi$
   implies the uniqueness of the equilibrium state for the weight $\phi$. Since this property holds
 by Claim 3 for all $\psi$
 such that $\norm{\psi}_{\log^q}<\infty$, 
 the proof is complete. 
\end{proof}

 In the second part of this work, we
will prove
 that, when $\phi$ and $\psi$
are H\"older continuous, the pressure function $P(t)$ defined above is actually analytic,
see  \cite[Theorem 1.3]{bd-eq-states-part2}.

We conclude this section with the following properties of the
equilibrium
state $\mu_\phi$
that
we will use in the next section.

\begin{proposition}\label{p:large-entropy}
Under the assumptions of Theorem \ref{t:main}, 
the metric entropy $\Ent{f}{\mu_\phi}$
of $\mu_\phi$ is strictly larger than $(k-1)\log d$.
 In particular, $\mu_\phi$ has no mass on proper analytic subsets of $\P^k$, its Lyapunov
  exponents are strictly positive and at least equal to ${1\over 2}(\Ent{f}{\mu_\phi} -  (k-1)\log d)$,
  and the function $\log |\jac Df|$ is integrable with respect to $\mu_\phi$.
   Moreover, 
  the Hausdorff dimension of $\mu_\phi$ satisfies
  \[
\dim_H  (\mu_\phi) \geq \frac{ (k-1)\log d}{\lam_1} + \frac{\Ent{f}{\mu_\phi} -  (k-1)\log d}{\lam_k}\cdot
  \]
\end{proposition}

\proof
Since $\mu_\phi$ maximizes the pressure and $\Ent{f}{\mu}=k\log d$, we have 
$$\Ent{f}{\mu_\phi} +\langle\mu_\phi,\phi\rangle \geq \Ent{f}{\mu}+\langle\mu,\phi\rangle \geq k\log d+ \min \phi.$$
Since by assumption we have $\Omega (\phi)<\log d$, it follows that 
$$\Ent{f}{\mu_\phi} \geq k\log d +\min\phi- \langle\mu_\phi,\phi\rangle  \geq k\log d -\Omega(\phi)>(k-1)\log d.$$
The Lyapunov exponents
of every ergodic invariant probability measure satisfying this property
are bounded below as in the statement, and in particular
the function $\log |\jac|$ is integrable
with respect to it,
see de Th\'elin \cite{de2008exposants} and Dupont \cite{dupont2012large}.
The bound on the Hausdorff dimension 
of $\mu_\phi$  is then a consequence of \cite{dupont2011dimension}, see also \cite{de2015measures}.

Let now $X$ be a proper analytic subset of $\P^k$. Assume by contradiction
that $m:=\mu_\phi(X)>0$. We choose such an $X$ which is irreducible and
of minimal dimension $p$. So, for all $n\geq 0$,  $f^n(X)$ is also an irreducible analytic set of dimension $p$.
We have 
\[\mu_\phi(f^n(X))=\mu_\phi(f^{-n}(f^n(X)))\geq \mu_\phi(X)=m.\]
It follows that $\mu_\phi ( f^n(X)\cap f^{n'}(X)) >0$ for some $n'>n\geq 0$. 
The minimality of the dimension $p$ implies that $f^n(X)=f^{n'}(X)$.

Replacing
$X,f$, and $\phi$ 
by
$f^n(X), f^{n'-n}$, and $\phi+\cdots+\phi\circ f^{n'-n-1}$
 we can assume that $X$ is
invariant and $\mu_\phi(X)>0$. Since $\mu_\phi$ is mixing, it is ergodic. We then deduce that $\mu_\phi(X)=1$.
Therefore, the metric entropy of $\mu_\phi$ is smaller than the topological 
entropy of $f$ on $X$. But this is a contradiction because the last one is at
most equal to $p\log d$, see \cite[Th.\,1.108 and Ex.\,1.122]{dinh2010dynamics}. The result follows.
\endproof

\subsection{Equidistribution of periodic points and end of the proof of Theorem \ref{t:main}}\label{s:pern}

Because of Proposition \ref{p:cv-m-phi}, 
Corollary \ref{c:equi-preim}, Lemma \ref{l:def-mu-phi}, 
and Propositions
 \ref{p:mixing} and 
 \ref{p:pu-uz}, to prove Theorem
\ref{t:main}
it only remains to establish
the equidistribution of (weighted) repelling periodic points of period
$n$ with respect to $\mu_\phi$, as $n\to \infty$.

\begin{theorem}\label{t:equidistr-pern}
Let $f\colon \P^k \to \P^k$
 be a holomorphic endomorphism
 of $\P^k$
  of algebraic
  degree $d\geq 2$ and satisfying Assumption {\bf (A)}. Let $\phi\colon \P^k \to \R$
 satisfy 
 $\norm{\phi}_{\log^q}<\infty$ for some $q>2$ and $\Omega (\phi)<\log d$.
  Let $\mu_\phi$ be the
 unique equilibrium state associated to $\phi$, and $\lam$ the scaling ratio. 
Then
for every $n\in \N$
there exists a
set $P'_n$
 of 
  repelling periodic
points of period $n$
 in the small Julia set
 such that
 \begin{equation}\label{eq:pern:goal}
\lim_{n\to \infty}
 \lam^{-n} \sum_{y \in P'_n} e^{\phi(y) + \phi(f(y)) + \dots + \phi (f^{n-1} (y))} \delta_y
 = \mu_\phi.
 \end{equation}
 \end{theorem}

Note that a related
equidistribution property
   for 
    H\"older continuous
    weights
was proved  by Comman-River-Letelier \cite{comman2011large}
for (hyperbolic and) topologically Collect-Eckmann rational maps on $\P^1$.

 To prove Theorem \ref{t:equidistr-pern},
we follow a now classical 
strategy 
due to Briend-Duval \cite{briend1999exposants} for the measure of maximal entropy
(which corresponds to the case $\phi\equiv 0$).
We 
employ
a trick due to X.\ Buff which simplifies the original proof.
An extra difficulty with respect to the case $\phi\equiv 0$ is due to
the fact that there is no a
priori upper bound for the mass of the left hand side of \eqref{eq:pern:goal} when $P'_n$
is replaced by the set of all repelling periodic
points of period $n$.

Given any point $x\in\P^k$
 we denote by
$\mu_{x,n}$ the measure
\[
\mu_{x,n} := \lam^{-n} \rho(x)^{-1} \sum_{f^n (a)=x } e^{\phi(a) + \phi(f(a)) + \dots + \phi (f^{n-1} (a))} \rho(a) \delta_a.
\]
It follows from
Corollary \ref{c:equi-preim}
that, for every continuous function $g\colon \P^k \to \R$, we have
\[
\sca{\mu_{x,n},g}
=
 \lam^{-n} \rho(x)^{-1} \sum_{f^n (a)=x } e^{\phi(a) + \phi(f(a)) + \dots + \phi (f^{n-1} (a))} \rho(a)g(a)
 \to
 \rho(x)^{-1} \sca{\rho(x)m_\phi, \rho g} = \sca{\mu_\phi, g}
\]
as $n\to \infty$. This means that, for all $x \in \P^k$,
we have
$\mu_{x,n}\to \mu_\phi$ as $n\to \infty$.

We denote by $0<L_1\leq \dots \leq L_k$ the 
Lyapunov exponents of $\mu_\phi$, see 
Proposition \ref{p:large-entropy}. 
We fix in what follows  a constant
$0<L_0 < L_1$.
Given $x\in X$,
 a ball $B$ of center $x$,  and $n\in \N$,
we
say that $g\colon B \to B'$
 is an $m$-\emph{good inverse branch of $f$ of order $n$ on $B$}
if 
\[g \circ f^n=\id_{|B'}
\quad  \mbox{ and }
\quad \diam f^{l}(B')
\leq  e^{-m -(n-l)L_0} 
\mbox{ for all } 0\leq l \leq n.\]
Notice that the definition in particular implies that $\diam (B) \leq e^{-m}$.
We
denote by $\mu^{(m)}_{B,n}$
 the measure
\[
\mu^{(m)}_{B,n} := \lam^{-n}\rho(x)^{-1}
 \sum_{a=g(x)} e^{\phi(a) + \phi(f(a)) + \dots + \phi (f^{n-1} (a))}\rho(a) \delta_a,
\]
where the sum is taken on the
$m$-good inverse branches $g$ of $f$ of order $n$ on $B$.
Since we have $\mu^{(m)}_{B,n}\leq \mu_{x,n}$ for 
all $n\geq 0$,
 it follows that
any limit value $\mu'_B$ of the sequence $\big\{ \mu^{(m)}_{B,n}\big\}$ satisfies
 $\mu'_B \leq \mu_\phi$. 
 In particular, we have
 $\|\mu'_B\|\leq 1$.
 
 Given $m>1$
   we say that a ball $B$ 
   centred
   at $x$ is $m$-\emph{nice}
  if 
  \begin{enumerate}
  \item $\inf_B \rho > (1-1/m)\sup_B \rho$;
  \item $\big\|\mu^{(m)}_{B,n}\big\| \geq 1-1/m$ for every 
$n$ sufficiently large.
  \end{enumerate}
 Observe that 
 the
  second condition implies that
$\diam (B)\leq e^{-m}$ for every $m$-nice ball $B$. Moreover, 
 we have
  $\|\mu'_B\| \geq 1-1/m$ for every limit value $\mu'_B$ of the sequence $\mu^{(m)}_{B,n}$. 

\begin{lemma}\label{l:all-nice}
For $\mu_\phi$-almost every
 $x\in \P^k$, 
 every sufficiently small ball centred at $x$ is $m$-nice.
\end{lemma}

The proof of Lemma \ref{l:all-nice}
 is elementary but makes uses of the natural extension of the system
 $(\P^k, f, \mu_\phi)$, see for instance \cite[Sec.\ 10.4]{cornfeld2012ergodic}.
 We 
 denote by
  $X_0, C_f, PC_f$
   the small Julia set, the critical set 
   and the postcritical set
   $PC_f := \cup_{n\geq 0} f^n (C_f)$ of $f$, respectively. We also set $X:= X_0 \setminus \cup_{m\in \N} f^{-m} ( PC_f)$.
By Proposition \ref{p:large-entropy}
we have $\mu_\phi (f^{-m} (PC_f))=0$ for every $m\in \N$, hence $\mu(X)=1$.
 We denote by 
  $\hat X$ the 
 set
 \[
  \hat X := \set{\hat x  := (x_n)_{n \in \mathbb Z} \colon x_n \in X, f(x_n) = x_{n+1}},
 \]
 by $\pi_n: \hat x \mapsto x_n$ the natural projection
 from $\hat X$ to $X$ and by
$\hat f\colon \hat X \to \hat X$
the map
\[
\hat f (\dots, x_{-1}, x_0, x_1, \dots) :=
(\dots, f(x_{-1}), f(x_0), f (x_1), \dots)
= (\dots, x_{0}, x_1, x_2, \dots).\]
Observe that
$\pi_n \circ \hat f = f \circ \pi_n$
 for all $n\in \mathbb Z$.
Let us consider on $\hat X$ the $\sigma$-algebra $\hat { \mathcal {B}}$ generated by all 
\emph{cylinders}, i.e., the sets
of the form
\[
A_{n, B} :=\pi_n^{-1} (B)= \{\hat x \colon  x_n \in B\}
\mbox{ for } n\leq 0 \mbox{ and } B\subseteq \P^k \mbox{ a Borel set}\]
and 
set
\[
\hat \mu_\phi (A_{n,B}) := \mu_\phi (B) \mbox{ for all }  A_{n,B} \mbox{ as above}. 
\]
It follows from the invariance of $\mu_\phi$ 
and the fact that $x_n \in B$ if and only if $x_{n-m} \in f^{-m} (B)$ 
(with $m\geq 0$)
that 
$\hat \mu_\phi$ is well defined on the
collection of the sets $A_{n,B}$ and
\[
\hat \mu_\phi (A_{n,B}) = \hat \mu_\phi (A_{n-m,B}) \mbox{ for all } m\geq 0.
\]
Similarly, for every $m>0$ and  Borel sets
$B_0,B_{-1}, \dots, B_{-m}\subseteq \P^k$
we then have 
\[
\begin{aligned}
\hat \mu_\phi ( \{ \hat x \colon x_0  \in B_0, x_{-1} \in B_{-1},  & \dots, x_{-m} \in B_{-m} \} )\\
&  =
\hat \mu_\phi ( \{ \hat x \colon
x_{-m} \in f^{-m} (B_0) \cap f^{-(m-1)} (B_{-1}) \cap \dots \cap B_{-m}  
\} )\\
&=\mu_\phi (f^{-m} (B_0) \cap f^{-(m-1)} (B_{-1}) \cap \dots \cap B_{-m}).
\end{aligned}
 \]
We then extend $\hat \mu_\phi$ to a 
probability
measure, still denoted by $\hat \mu_\phi$, on $\hat {\mathcal B}$. Observe that
$\hat \mu_\phi$ is $\hat f$-invariant by construction and satisfies 
 $(\pi_0)_* \hat \mu_\phi = \mu_\phi$.

For $n>0$ we denote by $f_{\hat x}^{-n}$ the inverse branch of $f^n$ defined in a neighbourhood
 of $x_0$ and such that $f_{\hat x}^{-n} (x_0)= x_{-n}$. 
This branch exists for all $x_0 \in X$.  
We have the following lemma.

\begin{lemma}\label{l:b:02}
For every $0<L< L_1$
there exist
two measurable functions
 $\eta_L \colon \hat X \to (0,1]$
 and $S_L \colon \hat X \to (1,+\infty)$ 
 such that, for 
 $\hat \mu_\phi$-almost every $\hat x \in \hat X$,
 the map $f^{-n}_{\hat x}$ is defined on $\B_{\P^k}(x_0, {\eta_L (\hat x)})$
with
$\lip (f^{-n}_{\hat x})\leq S_L (\hat x)e^{-nL}$
for every $n \in \N$.
\end{lemma}

\begin{proof}[Sketch of proof]
The statement
 is a consequence of
Proposition \ref{p:large-entropy}.
A direct proof in the case $\phi=0$ is given in
 \cite[Sec.\ 2]{briend1999exposants}
 and
\cite[Thm.\ 1.4(3)]{berteloot2008normalization}. The 
 case $n=1$ comes from a (quantitative) application of the 
inverse mapping theorem, which is then iterated to get 
functions $\eta_L$
 and $S_L$ valid for all $n$. The main point in the proof
 is an application of the Birkhoff ergodic theorem to the function $\log |\jac Df|$. This function
 is
 integrable with respect to the measure of maximal entropy $\mu_0$, which has continuous potentials,
 because of the Chern-Levine-Nirenberg inequality
 \cite{chern1969intrinsic}.
Since
this function is integrable with respect to $\mu_\phi$ by Proposition
\ref{p:large-entropy},
 the same proof applies in our setting.
 \end{proof}

\begin{proof}[Proof of Lemma \ref{l:all-nice}]
Since $\rho$
is continuous
and strictly positive, 
we only need to check
that, for
$\mu_\phi$-almost every
 $x\in \P^k$, 
 every sufficiently small ball $B$
 centred at $x$ 
 satisfies 
 $\big\|\mu^{(m)}_{B,n}\big\| \geq 1-1/m$ 
 for every $n$ sufficiently large.
 
 Let us consider the 
 disintegration
 of the measure $\hat \mu_\phi$
 with respect to $\mu_\phi$
 and the projection $\pi_0$. 
 We denote by $\hat \mu_\phi^x$ the conditional measure
 on $\{x_0 =x\}$. 
 The measure $\hat \mu^x_\phi$ is uniquely
 defined for 
 $\mu_\phi$-almost all $x \in X$ and characterized by the identity
\[
\sca{\hat \mu_\phi, g} = \sca{\mu_\phi , u(x)}, \mbox{ where } u(x) := \sca{\hat \mu^x_\phi, g}
\] 
 for all 
bounded measurable 
 functions $g \colon \hat X \to \R$. 
Since $(\pi_0)_*\hat \mu_\phi = \mu_\phi$, $\hat \mu^x_\phi$
is a probability measure
for $\mu_\phi$-almost every $x$.

 We will need a more explicit
 description of the conditional measures $\hat \mu^x_\phi$.
 For $n> 0$ and $x\in X$ we
  consider the measure $\hat \mu^x_{n}$ on $\hat X$ defined as follows.
First, let us consider the projection $\hat X\to X^{n+1}$ given by
\[
\hat \pi^n := (\pi_{-n}, \dots, \pi_{-1}, \pi_0). 
\]   
  For every element $(y_{-n}, \dots, y_{0}) \in X^{n+1}$
  we choose
   a representative $\hat z \in \hat X$ such that $z_{j}=y_j$ for all $-n \leq j \leq 0$. For any given $y_0$
  and any  $n>0$
   we then have
  $d^{kn}$ 
  distinct
  such representatives, and we denote by $\hat Z_n$ their collection.
  We then set
 \[
\hat \mu^x_{n} := \lam^{-n} \rho(x)^{-1} \sum_{\hat z\in \hat Z_n \colon  z_0=x } e^{\phi(z_{-n}) + \phi(z_{-n+1}) + \dots + \phi (z_{-1})} \rho(z_{-n})\delta_{\hat z}.
 \]
Since 
 this is a finite sum,
 the measures $\hat \mu^x_{n}$
 are well defined on $\hat X$.

 \medskip
 
\noindent
{\bf Claim.} We have $\lim_{n\to \infty} \hat \mu^x_{n} = \hat \mu^x_{\phi}$ for
$\mu_\phi$-almost every
 $x \in X$.
 
\begin{proof}
It is enough to check the assertion on the cylinders
 $A_{-i,B}$
for $i\geq 0$ and $B\subseteq \P^k$ a Borel set.
  It is clear
  that, 
  for all $n>0$,
  we have $\hat \mu^x_{n} (A_{0,B}) = \delta_x (B)$, which implies that
  \[\int \hat \mu^x_n (A_{0,B}) \mu_\phi (x) = \int \delta_x (B) \mu_\phi(x) = \mu_\phi (B).\]
Moreover, for all $n> i$, 
using the invariance of $\rho$
by $\lam^{-1}\Ll$
we
have
\[\begin{aligned}
\hat \mu^x_{n} (A_{-i,B}) 
& = \hat \mu^x_{n} (A_{-i,B}\cap \pi^{-1}_0 (x)) \\
& =
 \lam^{-n} \rho(x)^{-1} 
 \sum_{\hat z\in \hat Z_n \colon  z_0=x } 
 e^{\phi(z_{-n}) + \phi(z_{-n+1}) + \dots + \phi (z_{-1})} \rho(z_{-n})\delta_{\hat z} (A_{-i,B})\\
 &=
 \lam^{-n} \rho(x)^{-1}
 \sum_{\hat z \in \hat Z_i \colon z_0 =x} 
 (\Ll^{n-i} \rho )(z_{-i})
 e^{\phi(z_{-i}) + \phi(z_{-i+1}) + \dots + \phi (z_{-1})} \delta_{\hat z} (A_{-i,B})\\
&= \lam^{-i} \rho(x)^{-1}
 \sum_{\hat z \in \hat Z_i \colon z_0 =x} 
\rho (z_{-i})
 e^{\phi(z_{-i}) + \phi(z_{-i+1}) + \dots + \phi (z_{-1})}
 \delta_{\hat z} (A_{-i,B})\\
 & = \hat \mu^x_{i} (A_{-i,B}).
\end{aligned}\]
In order 
to conclude it is enough to prove
that
\[
\int \hat \mu^x_{i} (A_{-i,B}) \mu_\phi(x) = \mu_\phi (B) \mbox{ for all } i>0.
\]
We have
\[
\begin{aligned}
\int  \hat \mu^x_{i} (A_{-i,B}) \mu_\phi(x)
&=
\int 
\Big(
\lam^{-i}\rho(x)^{-1} 
\sum_{\hat z \in \hat Z_i \colon z_0 =x} e^{\phi (z_{-i} ) + \phi (z_{-i+1})  + \dots + \phi (z_{-1})}
\rho (z_{-i} ) \delta_{\hat z} (A_{-i,B})
\Big)
\mu_\phi(x) \\
&=
\int 
\Big(
\lam^{-i}\rho(x)^{-1} 
\sum_{f^i (a)= x} e^{\phi(a ) + \phi (f(a))  + \dots + \phi (f^{i-1}(a))}
\rho (a )\1_B (a)
\Big)
\mu_\phi(x) \\
&= \Big\langle
\mu_\phi, \lam^{-i}\rho^{-1} f^i_* (e^{\phi + \phi \circ f + \dots + \phi \circ f^{i-1}} \rho \1_B)
\Big\rangle\\
& = \Big\langle  \frac{\rho e^{\phi + \phi \circ f + \dots +\phi \circ f^{i-1}}}{\lam^i (\rho \circ f^i)  }  (f^{i})^* \mu_\phi,  \1_B\Big\rangle
= \mu_\phi (B),
\end{aligned}
\]
where in the last step we used the fact that the
Jacobian of $\mu_\phi$
(i.e., the Radon-Nidokym 
derivative $\frac{f^* \mu_\phi}{\mu_\phi}$)
 is given by $\lam \rho^{-1}e^{-\phi} (\rho\circ f)$, which implies
that
\[(f^{i})^* \mu_\phi = \lam^i \rho^{-1}
e^{ -\sum_{j=0}^{i-1} \phi \circ f^j}
(\rho\circ f^i) \mu_\phi.\]
This completes the proof of the Claim.
\end{proof}

 Let us now fix an integer  $m>0$, a constant $L_0 < L < L_1$, and a second positive integer $\gamma$.
  For every integer $N>0$ we set 
\[ \hat X_N :=\big\{\hat x \in \hat X \colon
\eta_L (\hat x) \geq N^{-1}
\mbox{ and }
S_L (\hat x)\leq N\big\}.
\]
 Observe that $\hat \mu_\phi (\hat X_N) \to 1$ as $N\to \infty$. In particular, there exists $N_0 = N_0 (m,\gamma)$
 such that,
 for every $N>N_0$, we have
 $\hat \mu_\phi (\hat X_N) > 1-1/(2m^{\gamma+1})$. It follows by Markov inequality
 that there exists a subset $X_\gamma \subset X$
 with 
 $\mu_\phi (X_\gamma) > 1- 1/m^\gamma$
  such that, for all $N> N_0$,
 \[
  \hat \mu^x_\phi (\hat X_N \cap \{ x_0=x\})>1-1/(2m) \mbox{ for all } x \in X_\gamma.
 \]

It is enough to prove the 
property in the lemma
 for all $x \in X_\gamma$. Let us fix one such $x$. By Lemma \ref{l:b:02} and the definition of $\hat X_N$,
  for every $\hat x \in \hat X_N$
  and $n\geq 0$
   the inverse branch
  $f^{-n}_{\hat x}$ is defined on the ball $B_{\P^k} (x_0, N^{-1})$
  with 
 $\lip (f^{-n}_{\hat x}) \leq N e^{-nL}$.
  In particular,
$\diam (f^{-n}_{\hat x} 
(B_{\P^k} (x_0, e^{-m}/(2N) ))) \leq e^{-m- n L_0}$  
   for all 
  $n\geq 0$.
 It follows that
 all inverse branches
on $B_{\P^k} (x, e^{-m}/(2N))$ 
corresponding to elements $\hat x \in \hat X_N \cap \{ x_0 =x\}$ are $m$-good
 for all $n$.
 
 The Claim above implies that
 \[
  \hat \mu^x_{n} (\hat X_N \cap \{ x_0=x\})>1-1/m \mbox{ for all } n \mbox{ large enough}.
 \]
 This precisely means that, for all $n$ sufficiently large,  
we have  $\big\|\mu^{(m)}_{B,n}\big\|> 1-1/m$, where $B=B_{\P^k} (x,e^{-m}/(2N) )$. 
This implies that such a ball
 $B$ is $m$-nice.
The proof is complete.
 \end{proof}
 
\begin{lemma}\label{l:series-phi}
There exists a positive constant $C=C(L_0,q)$ such that, 
for all $n\in \N, m>0$,
and
every $m$-good inverse branch
$g\colon B \to B'$
 of $f$ of order $n$ 
  on a ball $B$, 
 and for all
 sequences of points $\{x_l\},\{y_l\}$ with $0\leq l \leq n-1$
  and $x_l,y_l \in f^l (B')$ we have
\[
\sum_{l=0}^{n-1} \abs{ \phi (x_l) - \phi (y_l)}\leq C m^{-(q-1)}.
\]
\end{lemma}

\begin{proof}
Since  $g$ is $m$-good, 
we have 
$\dist(x_l,y_l) \leq 
 e^{-m-(n-l)L_0}$
 for all $0\leq l \leq n-1$.
  Hence,
\[
\sum_{l=0}^{n-1} \abs{ \phi (x_l) - \phi (y_l)}
 \leq 
\sum_{l=0}^{n-1} \norm{\phi}_{\log^q} |\log^\star \dist (x_l, y_l)|^{-q}
 \leq
 \norm{\phi}_{\log^q}
\sum_{l=1}^{\infty} | 1+m + l L_0 |^{-q}
\lesssim m^{-(q-1)},
\]
where the implicit constant depends on $L_0,q$ 
and we used the assumption that $q>2$.
\end{proof}

\begin{lemma}\label{l:bound-weak}
Let 
$\mathcal U$ 
be a finite collection of 
disjoint
 open subsets
of $\P^k$. 
For every $m>0$ 
 there exists
$n(m, \mathcal U)>m$ 
and,
 for every $n\geq n(m, \mathcal U)$,
  a set $Q_{m,n}$
of 
 repelling periodic points 
of period $n$ in the intersection of 
the union of the sets in $\mathcal U$ with the small Julia set
such that, for all $U \in \mathcal U$,
\[
\begin{aligned}
(1-1/m) 
\mu_\phi (U)  \leq 
\lam^{-n} \sum_{y \in Q_{m,n} \cap U} e^{\phi(y) + \phi(f(y)) + \dots + \phi (f^{n-1} (y))} 
\leq (1+1/m)
\mu_\phi(U).
\end{aligned}
\]
 \end{lemma}

  \begin{proof}
  We can assume that $\mathcal U$ consists
  of a single open set $U$, the general case follows
  by taking $n(m, \mathcal U)$ to be 
  the maximum of the $n(m,U)$, for $U\in \mathcal U$.
  We can also assume that $\mu_\phi (U)>0$
  because otherwise we can choose
 $n(m,U)=m+1$ and
 $Q_{m,n}= \varnothing$. 
Fix integers
 $ m_2 \gg m_1 \gg m$.
By Lemma \ref{l:all-nice}, for 
$\mu_\phi$-almost every
 point $a$, every
 ball
  of sufficiently small
radius centred at $a$ is 
$m_2$-nice. 
Hence, we can find a finite family of
 disjoint  $m_2$-nice balls $B_i\Subset U$, 
  such that $\mu_\phi (U \setminus \cup  B_i) < \mu_\phi (U) /m_2$.
  It is then enough to prove 
  the lemma
  for each $B_i$ instead of $U$.
  More precisely,
   let  $B= B_{\P^k}(a,r)$ be an $m_2$-nice ball.
   It is enough to find
   an $n(m_2)>m_2$ and, for all $n\geq n(m_2)$,
 a set $Q$ of repelling periodic points of period $n$ in $B\cap \supp (\mu_\phi)$
   such that
   \begin{equation}\label{e:goal-pern}
(1-1/m_1)
\mu_\phi (B) 
 \leq 
\lam^{-n} \sum_{y \in Q} e^{\phi(y) + \phi(f(y)) + \dots + \phi (f^{n-1} (y))} 
\leq (1+1/m_1)
 \mu_\phi(B).
   \end{equation}

   We fix in what follows
   an integer $m_3 \gg m_2/\mu_\phi(B)$ and
a second ball $B^\star =  B_{\P^k} (a, r^\star)$, with $r^\star<r$, such that 
$\mu_\phi (B^\star) > (1-1/m_2) \mu_\phi(B)$.
Choose
a finite family of disjoint
$m_3$-nice balls
$D_i$
 with the property that $\mu_\phi (\cup D_i)> 1-1/m_3$. 
We set
$D:= \cup D_i$ 
and let
$b_i$ be the center of $D_i$. 
We also fix balls $D^\star_i\Subset D_i$ centred at 
$b_i$ and such that $\mu_\phi (\cup  D^\star_i)> 1-1/m_3$ 
and set $D^\star := \cup D^\star_i$.

\medskip\noindent
{\bf Claim 1.}
There is an integer $M_1= M_1(m_2, B, B^\star, D_i)$
such that, for all 
$N\geq M_1$,
we have
\begin{equation}\label{e:mdn}
 (1-4/m_2) 
  \mu_\phi (B)\leq
\mu^{(m_3)}_{D_i,N} (B^\star)
 \leq (1+4/m_2)
  \mu_\phi (B) 
  \quad \mbox{ for all } i.
\end{equation}
\begin{proof}
Since the balls $D_i$ are $m_3$-nice and $m_3 \gg m_2/\mu_\phi(B)$,
 for every $i$ we have
\[
\big\|\mu^{(m_3)}_{D_i, N} \big\| \geq (1-\mu_\phi(B)/m_2)
  \mbox{ for all } N
  \mbox{ large enough.}
\]
Hence, 
since
$\mu^{(m_3)}_{D_i,N}\leq \mu_{b_i,N}$
and
$\norm{\mu_{b_i, N}} \leq 1+ o(1)$, 
we have $\big\| \mu_{b_i, N} - \mu^{(m_3)}_{D_i,N} \big\|\leq \mu_\phi (B)/m_2 + o(1)$.
Therefore,
in order to prove the claim
 it is enough to show that
\[
 (1-2/m_2)
\mu_\phi (B)\leq
\mu_{b_i,N} (B^\star) \leq (1+2/m_2)
 \mu_\phi (B)
\]
 for all $i$ 
 and all
 $N$ large enough. 
 This is a consequence of 
  Corollary \ref{c:equi-preim}
 and of the inequality $\mu_\phi (B^\star)>(1-1/m_2)\mu_\phi (B)$.
\end{proof}

Similarly,
 we also have the following.

\medskip\noindent
{\bf Claim 2.}
There is an integer $M_2=M_2(m_2, B, D^\star)$
such that, for all $N \geq M_2$, 
we have
\begin{equation}\label{e:mbn-tot}
1-4/m_2
 \leq
\mu^{(m_2)}_{B,N} (D^\star)
\leq 1+4/m_2.
\end{equation}

\begin{proof}
Since the ball $B$ is $m_2$-nice, 
 we have
\[
\big\|\mu^{(m_2)}_{B, N} \big\| \geq (1-1/m_2)  \mbox{ for 
all } N 
\mbox{ large enough}.
\]
Hence, by the 
 fact that $\mu^{(m_2)}_{B,N}\leq \mu_{a,N}$
 and $\norm{\mu_{a,N}}\leq 1 + o(1)$,
  in order to prove the claim
 it is enough to show that
\[
1-2/m_2
\leq 
\mu_{a,N} (D^\star)
\leq  
1+ 2/m_2
\]
for 
all $N$ large enough. 
This is again a
consequence of 
Corollary \ref{c:equi-preim}
and
of the inequality $\mu_\phi (D^\star) > (1-1/m_3)$.
\end{proof}
For every $N_1$
 sufficiently
large, 
every point in the support of $\1_{B^\star}\mu^{(m_3)}_{D_i,N_1}$ corresponds to 
an
$m_3$-good
inverse branch of $f$ of order ${N_1}$ mapping 
 $D_i$ to a relatively compact subset of
  $B$. Similarly, for every $N_2$ sufficiently large
every point in the support of $\1_{D^\star}\mu^{(m_2)}_{B,N_2}$ 
corresponds to 
an
$m_2$-good
inverse branch
of $f$ of order ${N_2}$ mapping
$B$ to a relatively compact subset of $D$. Composing such inverse branches we get 
inverse
branches $g_j$ of $f^{N_1+N_2}$ 
defined on $B$ 
whose images 
are
 relatively compact in $B$. 
  In what follows, we only consider these inverse branches $g_j$.
We also write $g_j$ as
$g^{(1)}_j \circ g^{(2)}_j$, where
$g^{(2)}_j$ is the corresponding 
inverse branch of $f^{N_2}$ on $B$
(whose image is then in $D$)
and
$g^{(1)}_j$ 
is the corresponding inverse branch of $f^{N_1}$ on $g^{(2)}_j (B)$. 
We also set $i=i(j)$, where
 $g^{(2)}_j (B) \subset D_i$.

Each inverse branch $g_j$
 as above 
contracts the Kobayashi metric of $B$, and thus admits a unique
 fixed
point $y_j$, which is
 attracting for $g_j$ and hence repelling for
 $f^{N_1+N_2}$. 
Up to possibly increasing the integers 
 $M_1$ and $M_2$
 given by the Claims above, 
we can assume that 
the above properties hold for $N_1 = M_1$ and $N_2 = M_2$.  
We set $n(m):=M_1(m_2)+M_2(m_2)$
for a fixed choice of sufficiently
large  $m_1, m_2,m_3$
and,
for all $n\geq n(m)$,
we define the set $Q$ as the union of all such fixed points constructed as above
with $N_1 = M_1 (m_2)$ and $N_2 = n- N_1 \geq M_2 (m_2)$. 
 The points in $Q$ are then repelling periodic points
of period $n=N_1+N_2$ for $f$.
Observe
that, for all $j$ and all $z\in B$, 
since $g_j (B)\Subset B$ we have $g_j^l (z)\to y_j$ as $l\to\infty$. Since $B$ intersects
the small Julia set, 
by taking $z$ in the small Julia set we see 
that
 $y_j$ belongs to the small Julia set.
To conclude, we need to prove
\eqref{e:goal-pern} for this choice of $Q$. 
We set 
 \[\mu_n := \lam^{-n} \sum_{y\in Q} e^{\phi (y ) + \phi (f(y)) + \dots +\phi (f^{n-1} (y))}\delta_y
 =
  \sum_{j} e^{\phi (y ) + \phi (f(y)) + \dots +\phi (f^{n-1} (y))}\delta_{y_j}
 \]
and
\[
\begin{aligned}
\tilde \mu_n :=  \lam^{-n}
\sum_{j} 
& \Big(
e^{\phi (g^{(1)}_j (b_{i(j)}) ) + \phi ( f\circ g^{(1)}_j ( b_{i(j)} )) + \dots +\phi (f^{N_1-1} \circ g^{(1)}_j (b_{i(j)}))}
\frac{\rho (g^{(1)}_j (b_{i(j)})) }{\rho( b_{i(j)})} \cdot \\
& \cdot e^{\phi (g^{(2)}_j (a) ) + \phi ( f\circ g^{(2)}_j ( a )) + \dots +\phi (f^{N_2-1} \circ g^{(2)}_j (a))}
\frac{\rho(g_j^{(2)} (a))}{\rho(a)}
\delta_{g_j (a)}\Big).
\end{aligned}
\]
Observe that
there is a correspondence between the
terms in $\mu_n$ and those in $\tilde \mu_n$.
Moreover, since all the balls $B$ and $D_i$ are $m_2$-nice,
we have 
\[| \rho (g^{(1)}_j (b_{i(j)}))  /\rho (a)-1| \lesssim m_2^{-1} \mbox{ and }
|\rho (g^{(2)}_j (a) ) / \rho( b_{i(j)})-1| \lesssim m_2^{-1}
\mbox{ for all  } i \mbox{ and } j.\]
It
 follows from these inequalities and Lemma \ref{l:series-phi}
  that
$\abs{\mu_n (B) - \tilde \mu_n (B) }\lesssim \tilde \mu_n (B) m_2^{-1}$.
 Hence, in order to conclude it is enough to prove
that
\[
(1-1/(2m_1)) 
\mu_\phi (B)  \leq 
\tilde \mu_n (B)
\leq (1+1/(2m_1))
  \mu_\phi (B)
\]
because
$m_2$ is chosen
 large enough. By construction,
 we have
\[
\tilde \mu_n (B) = 
\sum_i
\mu^{(m_2)}_{B,N_2} (D^\star_i) \cdot \mu^{(m_3)}_{D_i,N_1} (B^\star).
\]
By Claim 1, this implies that
\[
 (1-4/m_2)
\mu_\phi (B)
 \sum_{i} 
\mu^{(m_2)}_{B,N_2} (D^\star_i)
\leq
 \tilde \mu_n (B) \leq
  (1+4/m_2)
  \mu_\phi (B)
\sum_{i} 
\mu^{(m_2)}_{B,N_2} (D^\star_i).
\]
The assertion then follows from
Claim 2 and the fact that $\sum_i \mu^{(m_2)}_{B,N_2} (D^\star_i) = \mu^{(m_2)}_{B,N_2} (D^\star)$,
 by taking $m_2$ large enough.
 \end{proof}
 
 We can now
 conclude the proof of Theorem \ref{t:equidistr-pern}. As mentioned at the beginning of the section, this also completes
 the proof of Theorem \ref{t:main}.

\begin{proof}[End of the proof of Theorem \ref{t:equidistr-pern}]
For every $i\in \N$
we construct 
a finite family 
of disjoint
open sets
$\mathcal U_i :=\{U_{i,j}\}_{1\leq j \leq J_i}$
with the following properties:
\begin{enumerate}
\item 
 $\mu_\phi 
(\cup_{1\leq j \leq J_i} U_{i,j})=1$;
\item for all $1\leq j \leq J_i$ we have $\diam (U_{i,j}) < 1/i$;
\item for all $i\geq 2$ and $1\leq j \leq J_i$ there exists $1\leq j' \leq J_{i-1}$
such that
$U_{i,j}\subset U_{i-1, j'}$.
\end{enumerate}
We can construct these sets using local coordinates and generic real hyperplanes which are parallel to the coordinate hyperplanes.
Observe also that, by the first condition, we have $\mu_\phi (\partial U_{i,j})=0$ for all $i$ and $1\leq j \leq J_i$.

For every $n$, we define $i_n:= \max\{ m\leq n \colon n \geq
n(m, \mathcal U_m)
\}$, 
where $n(m, \mathcal U_m)$
is given by Lemma \ref{l:bound-weak}. 
Observe that $i_n \to \infty$ as $n\to \infty$.
We define $P'_n
\subset 
\cup_j U_{i_n, j}
$
as the union of the
 sets of repelling
 periodic
points of period $n$
in the small Julia set
 obtained
 by applying Lemma
 \ref{l:bound-weak} to 
 the collection $\mathcal U_{i_n}$
 instead of $\mathcal U$,
and set 
\[\mu'_n := \lam^{-n} \sum_{y \in P'_n} e^{\phi(y) + \phi(f(y)) + \dots + \phi (f^{n-1} (y))} \delta_y.
\]
By Properties (i) and (ii) 
of the open sets $U_{i,j}$ and Lemma \ref{l:bound-weak},
 any limit $\mu'$ of the sequence $\{\mu'_n\}$
has mass 1.
 So, since $\mu_\phi (
\cup_j U_{i_n,j} 
 )=1$
 for all $n$
  and $\diam (U_{i,j})<1/i$ for all $i$,
  it is enough to prove that
\begin{equation}\label{e:mm_ij}
\liminf_{n\to \infty} \mu'_n (U_{i^\star,j^\star}) 
\geq  \mu_\phi (U_{i^\star, j^\star}) \mbox{ for all } i^\star \in \N \mbox{ and } 1 \leq j^\star \leq J_{i^\star}.
\end{equation}
Indeed,
given any open set $A\subseteq \P^k$, we can write $A$ as a
countable
 union
of compact 
sets of the form 
$\bar U_{i,j}\Subset A$, 
 overlapping only on their boundaries. We then see
that 
\eqref{e:mm_ij}
implies that $\mu_\phi(A)\leq  \mu' (A)$ for every open set $A$,
and
 the facts 
 that $\|\mu_\phi\|= \|\mu'\|$ 
 and
 $\mu_\phi(\partial U_{i,j})=0$
for all $i,j$
 imply that
$\mu_\phi=\mu'$.

We can then fix $i^\star,j^\star$
as in \eqref{e:mm_ij}
and a positive number $\eps$, 
and
it is enough to
prove that
 \begin{equation*}\label{e:diff-eps}
\mu'_n (U_{i^\star, j^\star}) \geq \mu_\phi (U_{i^\star, j^\star}) - \eps
\quad \mbox{ for all } n \mbox{ sufficiently large}.
\end{equation*}

We only consider in what follows integers $n$ such that $i_n > i^\star$
 and the sets $U_{i_n,j}$ which
are contained in
$U_{i^\star, j^\star}$. 
For all such $n$, we have $\mu_\phi (U_{i^\star, j^\star}) =
\sum_{ j}\mu_\phi (U_{i_n,j})$
and $\mu'_n (U_{i^\star, j^\star}) =  \sum_{j} \mu'_n (U_{i_n,j})$. It follows
by
the definition of $\mu'_n$ and
Lemma \ref{l:bound-weak}
 that
\[
\big|
\mu'_n (U_{i^\star, j^\star}) - \mu_\phi (U_{i^\star, j^\star})
\big|
 \leq  \sum_j
\abs{   \mu'_n (U_{i_n,j}) -\mu_\phi (U_{i_n,j}) } 
 \leq 
 i_n^{-1}\sum_j 
\mu_\phi (U_{i_n, j}) 
=
 i_n^{-1} \mu_\phi (U).
\]
The assertion follows.
\end{proof}

\begin{remark}
One could improve the argument in the proof of Lemma 
\ref{l:bound-weak}
to obtain that $P'_n$ can be taken to
be a subset of the repelling periodic points with
a good control of
 the eigenvalues, see
for instance \cite{berteloot2008normalization,berteloot2019distortion}. 
This implies that,
setting $\Sigma_j := L_{k-j+1} + \dots +L_k$,  we have
\begin{equation*}\label{e:lyap-partial}
\Sigma_j = 
\lim_{n\to \infty}
\lam^{-n} \sum_{y \in P'_n}
 \frac{1}{n}
 e^{\phi(y) + \phi(f(y)) + \dots + \phi (f^{n-1} (y))}
\textstyle
\log  \Big\| \bigwedge^{j} Df^n_y \Big\|
\end{equation*} 
 and,
 in particular,
 \begin{equation*}\label{e:lyap}
\Sigma_k = \sum_{j=1}^k L_j
=
\lim_{n\to \infty}
\lam^{-n} \sum_{y \in P'_n}
 e^{\phi(y) + \phi(f(y)) + \dots + \phi (f^{n-1} (y))}
\log |\jac Df_y|.
 \end{equation*}

Here,
$Df^n_x \colon T_x \P^k \to T_{f^n (x)}\P^k$ 
denotes
the differential of $f^n$
at $x$. This is a linear map from the
complex tangent space
of $\P^k$ at $x$ to the one 
at $f^n (x)$. It induces the natural 
 linear map
$\bigwedge^j Df^n_x$ from the exterior power $\bigwedge^j T_x\P^k$
to
$\bigwedge^jT_{f^n (x)}\P^k$.
\end{remark}

\printbibliography

\end{document}